\documentclass[12pt,centertags,oneside]{amsart}

\usepackage{amsmath,amstext,amsthm,amssymb,amsfonts,amscd}
\usepackage{mathrsfs}
\usepackage[colorlinks=true]{hyperref}
\usepackage{graphicx}
\usepackage[T1]{fontenc}
\usepackage[latin1]{inputenc}
\usepackage{color,tocvsec2}
\usepackage{typearea}
\usepackage{enumerate}
\usepackage{lscape}
\usepackage[all]{xy}
\usepackage{fouriernc}

\usepackage{multicol}        
\makeindex             

\usepackage[a4paper,width=14.66cm,top=2.52cm,bottom=2.52cm]{geometry}

\theoremstyle{plain}
\newtheorem{lem}{Lemme}[section]
\newtheorem{thm}[lem]{Theorem}
\newtheorem{prop}[lem]{Proposition}
\newtheorem{cor}[lem]{Corollary}
\newtheorem{ass}[lem]{Assumption}

\theoremstyle{definition}
\newtheorem{defin}[lem]{Definition}

\theoremstyle{remark}
\newtheorem{re}[lem]{Remark}

\numberwithin{equation}{section}
\numberwithin{figure}{section}

\newcommand{\cE}{\mathcal{E}}
\newcommand{\cF}{\mathcal{F}}

\newcommand{\cO}{\mathcal{O}}

\newcommand{\cX}{\mathcal{X}}

\newcommand{\bR}{\mathbf{R}}
\newcommand{\bC}{\mathbf{C}}

\newcommand{\fb}{\mathfrak{b}}
\newcommand{\fg}{\mathfrak{g}}
\newcommand{\fh}{\mathfrak{h}}
\newcommand{\fk}{\mathfrak{k}}
\newcommand{\fm}{\mathfrak{m}}
\newcommand{\fn}{\mathfrak{n}}
\newcommand{\fp}{\mathfrak{p}}

\newcommand{\ft}{\mathfrak{t}}
\newcommand{\fz}{\mathfrak{z}}
\newcommand{\fu}{\mathfrak{u}}
\newcommand{\fa}{\mathfrak{a}}

\DeclareMathOperator{\Tr}{\mathrm Tr}
\DeclareMathOperator{\Ad}{\mathrm Ad}

\DeclareMathOperator{\vol}{\mathrm vol}

\DeclareMathOperator{\ad}{\mathrm ad}
\DeclareMathOperator{\Sp}{\mathrm{Sp}}

\renewcommand{\Re}{\mathrm{Re}\,}

\DeclareMathOperator{\End}{\mathrm End}

\DeclareMathOperator{\GL}{\mathrm GL}

\newcommand{\<}{\langle}
\renewcommand{\>}{\rangle}
\newcommand{\ol}{\overline}
\newcommand{\ul}{\underline}

\renewcommand{\(}{\left(}
\renewcommand{\)}{\right)}
\renewcommand{\[}{\left[}
\renewcommand{\]}{\right]}
\renewcommand{\l}{\leqslant}
\newcommand{\g}{\geqslant}
\newcommand{\e}{\epsilon}

\newcommand{\bbS}{\mathbb{S}}

\newcommand{\Trs}{{\rm Tr}_{\rm s}}

\begin{document}

\title[Analytic torsion and dynamical zeta function]{Analytic torsion, dynamical zeta function, and the Fried 
conjecture for admissible twists}
\author{Shu Shen}
 \address{Institut de Math\'ematiques de Jussieu-Paris Rive Gauche, 
Sorbonne Universit\'e,  4 place Jussieu, 75252 Paris Cedex 5, France.}
\email{shu.shen@imj-prg.fr}
\thanks {}

\makeatletter
\@namedef{subjclassname@2020}{%
  \textup{2020} Mathematics Subject Classification}
\makeatother

\subjclass[2020]{58J20,58J52,11F72,11M36,37C30}
\keywords{Index theory and related fixed point theorems, analytic 
torsion, Selberg trace formula, Ruelle dynamical zeta function}

\date{\today}

\dedicatory{}

\begin{abstract}
We show an equality between the analytic torsion and the absolute 
value at the zero point of the Ruelle dynamical zeta function on a 
closed odd dimensional locally symmetric space twisted by an 
acyclic flat vector bundle obtained by the restriction of a 
representation of the underlying Lie group. This generalises author's 
previous result  for unitarily flat vector bundles, and  the results 
of Br\"ocker, M\"uller, and Wotzke on closed hyperbolic manifolds. 
\end{abstract}

\maketitle
\tableofcontents

\settocdepth{section}
\section*{Introduction}
The purpose of this article is to study the relation between the 
analytic torsion and the  value at the zero point of the Ruelle dynamical zeta function associated to a  flat vector bundle, which is not 
necessarily unitary,  on a closed  odd dimensional  locally symmetric space of reductive type. 

Let $Z$ be a smooth closed manifold.   Let $F$ be a 
complex flat  vector bundle on $Z$.
Let $H^{\scriptscriptstyle\bullet }(Z,F)$ be the cohomology of sheaf of locally constant sections of $F$.
We assume $H^{\scriptscriptstyle\bullet }(Z,F)=0$.

%

Let $g^{TZ},g^{F}$ be  metrics  on $TZ$ and $F$.
The analytic torsion $T(F)$ of Ray-Singer \cite{RSTorsion} is 
a spectral invariant defined by a weighted product of zeta 
regularised determinants the Hodge Laplacian associated with $g^{TZ},g^F$. When $\dim Z$ is odd, they showed that $T(F)$ does not depend on the metric data. 

 Ray and Singer \cite{RSTorsion} conjectured, which was proved later  by Cheeger 
\cite{Ch79} and M\"uller \cite{Muller78},  that if  $F$ is unitarily flat (i.e., the holonomy representation of $F$ is unitary) 
the  analytic torsion coincides with  its topological counterpart, the Reidemeister 
torsion \cite{ReidemeisterTorsion,FranzTorsion,dRTorsion}. 
%
%
%
%
%
Bismut-Zhang \cite{BZ92} and M\"uller 
\cite{Muller2} simultaneously considered generalisations of this 
result. M\"uller \cite{Muller2} extended his result to 
odd dimensional oriented manifolds where only $\det F$ is required to 
be unitary. Bismut and Zhang \cite{BZ92}  
generalised the original Cheeger-M\"uller theorem to 
arbitrary flat vector bundles with arbitrary Hermitian metrics on a manifold with arbitrary dimension orientable or not.

Milnor \cite{MilnorZcover} initiated  the study of the relation 
between the  torsion invariant and a 
dynamical system. When $Z$ is an orientable hyperbolic 
manifold,  Fried \cite{FriedRealtorsion,Friedn2} showed an identity between the analytic torsion of an acyclic unitarily flat vector bundle and  the value at the zero point of the Ruelle dynamical zeta function of the 
geodesic flow of $Z$. He conjectured \cite[p.~66, Conjecture]{Friedconj} 
that similar results should hold true for more general flows. In 
\cite{Shfried}, following previous contributions by Moscovici-Stanton 
\cite{MStorsion}, using Bismut's orbital integral formula \cite{B09}, the author affirmed the Fried conjecture for
geodesic flows on closed odd dimensional\footnote{The case of even 
dimension is trivial \cite[Remark 5.12]{S19a} (c.f.  Remark 
\ref{re45}).} locally 
symmetric manifolds equipped with an acyclic unitarily flat vector bundle.  In \cite{Shen_Yu}, the 
authors made a further generalisation to closed locally symmetric 
orbifolds. We refer the reader to \cite{Ma_bourbaki} for an 
introduction to the technique used in \cite{Shfried}.

 When the flat vector bundle is not unitary, 
 M\"uller \cite{Muller20} and Spilioti 
 \cite{Spilioti18,Spilioti15,Spilioti20}  related the leading coefficients of the Laurent series of the Ruelle dynamical zeta function at the zero point to a weighted product of zeta regularised determinants of the flat Laplacian of Cappell-Miller 
\cite{CappellMiller} on orientable odd dimensional  hyperbolic 
manifolds. When the flat vector bundle is near to  an acyclic 
and unitary one, the authors have shown that  the Ruelle dynamical 
zeta function is regular at the zero point and its  value is equal to 
the complexed valued analytic torsion of Cappell-Miller  
\cite{CappellMiller}. In \cite{Shen20}, we generalised the above results to odd dimensional locally  symmetric spaces.

In this article,  we  prove the Fried conjecture on odd dimensional locally symmetric  spaces for a class\footnote{By Margulis' super-rigidity \cite[Section VII.5]{Margulis91} (see also 
\cite[Section XIII.4.6]{BW}), this is the most 
interesting case, when the real rank of the 
locally symmetric space is  $\ge 2$.} of flat 
vector bundles, which is not necessarily  close to a unitary one,  and whose holonomy representations are the restrictions of  representations of the 
underlying reductive groups. This  generalises the previous results 
of Br\"ocker \cite{Brocker98}, M\"uller \cite{Muller_3_torsion}, and 
Wotzke  \cite{Wotzke} on orientable odd dimensional hyperbolic manifolds. 

We refer the reader to \cite{ShenYuMorseSmale,ShVariationRuelle} for the Fried 
conjecture for the Morse-Smale flow and the Anosov flow, to \cite{S19a} for a survey  on the Fried conjecture.

Now, we will describe our results in more detail, and explain the techniques used in their proofs.


\subsection{The analytic torsion}
Let $Z$ \index{Z@$Z$}   be a smooth closed manifold, and let $F$\index{F@$F$} be a complex flat vector bundle on $Z$.

Let $g^{TZ}$ \index{G@$g^{TZ}$} be a Riemannian metric on $TZ$, and 
let $g^F$\index{G@$g^{F}$} be a Hermitian metric on $F$. To $g^{TZ}$ 
and $g^F$, we can associate an $L^2$-metric on 
$\Omega^{\scriptscriptstyle\bullet }(Z,F)$, the space of differential forms with 
values in $F$. Let  $\Box^Z$ \index{B@$\Box^Z$} be the  Hodge 
Laplacian acting on $\Omega^{\scriptscriptstyle\bullet }(Z,F)$. By  Hodge theory, we have a canonical isomorphism
\begin{align}\label{eq:Hodgein}
  \ker\Box^Z\simeq H^{\scriptscriptstyle\bullet }(Z,F).
\end{align}

The analytic torsion $T(F)$ is a positive real number defined by the following weighted 
product of the zeta regularised determinants (see Section \ref{SgLap})
\begin{align}\label{eq:inn}
T(F)=\prod_{i=1}^{\dim Z}\det\(\Box^Z|_{\Omega^i(Z,F)\cap \(\ker 
\Box^{Z}\)^{\bot}}\)^{(-1)^ii/2}.
\end{align}
By \cite{RSTorsion} and \cite[Theorem 0.1]{BZ92}, if $\dim Z$ is odd 
and if $H^{\scriptscriptstyle\bullet }(Z,F)=0$, then $T(F)$ is independent of $g^{TZ}$ and $g^{F}$. Therefore, it is a topological invariant. 

When $Z$ is a closed orbifold, the analytic torsion is still well 
defined \cite{Ma_Orbifold_immersion,Daiyu}. In 
\cite[Corollary 4.9]{Shen_Yu}, the authors show that if $Z$ as well 
as all the singular strata have odd dimension, then the analytic 
torsion of an acyclic orbifold flat vector bundle is still a topological invariant. 

\subsection{The Ruelle dynamical zeta function}
Let us recall the definition of the Ruelle dynamical zeta 
function  associated to a geodesic flow introduced by Fried \cite[Section 5]{Friedconj} (see also \cite[Section 2]{S19a}).

Let $(Z,g^{TZ})$ be a connected manifold with nonpositive sectional 
curvature. 
Let $\Gamma$\index{G@$\Gamma$} be the fundamental group of $Z$, and 
let $[\Gamma]$ be the set of the conjugacy classes of $\Gamma$. 
For $[\gamma]\in [\Gamma]$ \index{G@$[\gamma]$}, let 
$B_{[\gamma]}$\index{B@$B_{[\gamma]}$} be the set of closed 
geodesics in the free homotopy class associated to $[\gamma]$. It is easy to see that all the 
elements in $B_{[\gamma]}$ have the same length $\ell_{[\gamma]}$. 


For simplicity, 
assume that  all the $B_{[\gamma]}$ are  
smooth finite dimensional  submanifolds of the loop space of $Z$. 
This is the case if $(Z,g^{TZ})$ has a negative sectional curvature 
or if $Z$ is locally symmetric.  If $\gamma\neq 1$, the group 
$\bbS^1$ acts locally freely on  $B_{[\gamma]}$ by rotation, so that  
$ B_{[\gamma]}/\bbS^1$ is an orbifold. Let $\chi_{\rm orb}(B_{[\gamma]}/\bbS^1)\in \mathbf{Q}$ \index{C@$\chi_{\rm orb}$} be the orbifold Euler characteristic \cite{SatakeGaussB}. Denote by \index{M@$m_{[\gamma]}$}
\begin{align}
m_{[\gamma]}=\left|\ker\big(\bbS^1\to {\rm 
Diff}(B_{[\gamma]})\big)\right|\in \mathbf{N}^{*}
\end{align}
the multiplicity of a generic element in $B_{[\gamma]}$. Let 
$\e_{[\gamma]}=\pm1$ be the Lefschetz index of the Poincar\'e return 
map induced by the geodesic flow (see 
\cite[(2.17)]{S19a} for a precise definition). If $Z$ is locally 
symmetric, then $\e_{[\gamma]}=1$. 

If $r\in \mathbf N$, let $\rho:\Gamma\to \GL_{r}(\bC)$ be a representation of $\Gamma$.  The formal dynamical zeta function is defined for $\sigma\in \bC$ by
\begin{align}\label{eq:ZetaE}
R_{\rho}(\sigma)=\exp\(\sum_{[\gamma]\in [\Gamma_{+}]} 
\e_{[\gamma]}\Tr[\rho(\gamma) ]\frac{\chi_{\rm orb}(
B_{[\gamma]}/\bbS^1)}{m_{[\gamma]}}e^{-\sigma \ell_{[\gamma]}}\),
\end{align}
where $[\Gamma_{+}]= [\Gamma]-\{1\}$ is the set of the non trivial 
conjugacy classes of $\Gamma$. 
We will say that the formal dynamical zeta function is well defined if $R_\rho(\sigma)$ is  holomorphic for $\Re(\sigma)\gg1$ and extends meromorphically to $\sigma\in \bC$.

If $(Z,g^{TZ})$ has  negative sectional curvature,  the geodesic 
flow on the sphere bundle of  $(Z,g^{TZ})$ is Anosov. In this case, if 
$\rho$ is a trivial representation, $R_{\rho}(\sigma)$ has been shown to be well defined  by Giulietti-Liverani-Pollicott \cite{GLP2013} and Dyatlov-Zworski 
\cite{DyatlovZworski}. For general $\rho$, the proof of the 
meromorphic extension of $R_\rho$ is not particularly difficult.  For behaviour of the Ruelle zeta function near $\sigma=0$, we refer the reader to the work of Dyatlov and Zworski \cite{zworski_zero}, Dang, Guillarmou, Rivi\`ere, and Shen 
\cite{ShVariationRuelle}, as well as Borns-Weil and 
Shen \cite{BWShen20}.

%


\subsection{Results of Fried, Br\"ocker, Wotzke, and M\"uller  on hyperbolic manifolds} 
Assume that $Z$ is an odd dimensional   connected orientable closed hyperbolic manifold.
%
%
%
 Let $F$ be the unitarily flat vector bundle on $Z$ with holonomy  
 $\rho: \Gamma\to {\rm U}(r)$.
 
Using the Selberg trace formula, Fried  \cite[Theorem 
3]{FriedRealtorsion} showed that there exist explicit constants 
$C_\rho\in \bR^{*}$ and $r_\rho\in \mathbf{Z}$ such that as  $\sigma\to 0$,
\begin{align}
R_\rho(\sigma)=C_\rho T(F)^2\sigma^{r_\rho}+\cO(\sigma^{r_{\rho}+1}).
\end{align}
Moreover, if $H^{\scriptscriptstyle\bullet }(Z,F)=0$, then
\begin{align}
&C_\rho=1,&r_\rho=0,
\end{align}
so that
\begin{align}\label{eq:Fintro}
R_\rho(0)=  T(F)^2.
\end{align}

When $\rho$ is not unitary but a restriction of a representation of the orientation preserving isometric group of $Z$,  similar 
results (see Theorem \ref{Thm1}) have been shown by Br\"ocker \cite{Brocker98} and Wotzke 
\cite{Wotzke}, as well as M\"uller \cite[Theorem 
1.5]{Muller_3_torsion}. 

%

\subsection{The main result of the article}
Let $G$ \index{G@$G$} be a linear connected real reductive group \cite[p. 3]{Knappsemi}, and let $\theta$\index{T@$\theta$} be the  Cartan involution. Let
$K$\index{K@$K$} be the maximal compact subgroup of $G$ of the points of $G$ that are fixed by $\theta$.
Let $\fk$ \index{K@$\fk$}and  $\fg$\index{G@$\fg$} be the  Lie algebras of $K$ and $G$, and let $\fg=\fp\oplus \fk$ be the Cartan decomposition. Let $B$\index{B@$B$} be a nondegenerate bilinear symmetric form
on $\fg$ which is invariant under $G$ and $\theta$. Assume that $B$ is positive on $\fp$ and negative on
$\fk$. Set $X = G/K$.\index{X@$X$} Then $B|_{\fp}$ induces a Riemannian metric on $X$, which has nonpositive sectional curvature.

Let $\Gamma\subset G$\index{G@$\Gamma$} be a discrete  torsion free 
cocompact subgroup of $G$.  Let $Z=\Gamma\backslash X$ be the 
associated  locally symmetric manifold, which is equipped   with the 
induced Riemannian metric $g^{TZ}$.  Let $\rho:\Gamma\to \GL(E)$ be a 
finite dimensional  complex representation of $\Gamma$. Let 
$F$ 
be the associated flat  vector bundle on $Z$. In 
\cite[Theorem 0.1 i)]{Shen20}, we have shown that if $\dim Z$ is odd, the Ruelle zeta function $R_{\rho}$ has a 
meromorphic extension to $\bC$. 

Suppose now that  $\rho$ extends to a representation of $G$, which is 
still denoted by $\rho$.  
We assume also that $E$ has an admissible metric\footnote{If $G$ is 
semisimple or more generally if $G$ has a compact centre, then all the 
representations of $G$ has an admissible metric (\cite[Lemma 
3.1]{MatsuchimaMurakami}, Proposition \ref{proprepa}).}  $\<,\>_{E}$, i.e.,  
$\fp$ acts symmetrically and $\fk$ acts 
antisymmetrically on $(E,\<,\>_{E})$. By a construction due to Matsushima-Murakami \cite{MatsuchimaMurakami}, $\<,\>_{E}$ induces canonically a Hermitian metric $g^{F}$ on $F$  (see also Section \ref{sFadmetric}). 

%

%
Let $T(F)$ be the analytic torsion of $F$ associated to  $(g^{TZ},g^{F})$.  The following theorem generalises \cite[Theorem 0.1]{Shfried} where 
$\rho$ is assumed to be unitary, and Br\"ocker \cite{Brocker98}, 
M\"uller \cite{Muller_3_torsion}, 
and Wotzke \cite{Wotzke} where $Z$ is hyperbolic. 



\begin{thm}\label{Thm1}
	Assume that $\dim Z$ is odd and that  $\rho:G\to \GL(E)$ is a finite dimensional complex representation of $G$ with an admissible metric. 	Let $(F,g^{F})$ be
	the associated  Hermitian flat vector bundle. Then
	there exist  constants $C_{\rho}\in \bC^{*}$ 
	and $r_{\rho}\in \mathbf{Z}$ such that when $\sigma\to 0$, we have 
	\begin{align}\label{eq1rdii}
		R_{\rho}(\sigma)=
		C_{\rho}T(F)^{2}\sigma^{r_{\rho}}+\cO(\sigma^{r_{\rho}+1}).
	\end{align}
	Moreover, if $H^{\scriptscriptstyle\bullet }(Z,F)=0$, then 
	\begin{align}\label{eq2rdin}
&		\left|C_{\rho}\right|=1,&r_{\rho}=0,
	\end{align}
	so that 
	\begin{align}\label{eqRTin}
		\left|R_{\rho}(0)\right|=T(F)^{2}. 
	\end{align}
\end{thm}

Set $\rho^{\theta}=\rho\circ \theta$. Then $\rho^{\theta}$ is still 
a representation of $G$ with an admissible metric.  If $\rho\simeq \rho^{\theta}$, we will show in Theorem  \ref{thm1} that the constant  
$C_{\rho}\in \bR^{*}$ and we can remove the absolute 
value in \eqref{eq2rdin} and \eqref{eqRTin}. For general $\rho$, by 
\cite[Theorem 0.1 ii) iii)]{Shen20}, the  argument of $C_{\rho}$ is 
determined by the argument of a \eqref{eq:inn}-like product of zeta 
regularised determinants of the flat 
Laplacian of Cappell and Miller, which is related to the complex 
valued analytic torsion of Cappell and Miller \cite{CappellMiller}. 

When $\rho$ is irreducible and $\rho\not\simeq\rho^{\theta}$, thanks 
to the vanishing of the cohomology $H^{\scriptscriptstyle\bullet }(Z,F)$ \cite[Theorem 
	VII.6.7]{BW}, we have the following corollary. 


\begin{cor}\label{cori1}
	Assume that $\dim Z$ is odd and that  $\rho:G\to \GL(E)$ is a finite dimensional complex representation of $G$ with an admissible metric. Let $(F,g^{F})$ be
	the associated   Hermitian flat vector bundle. 
	Suppose  that  $\rho$	is  irreducible and  $\rho\not\simeq 
	\rho^\theta$. Then, $R_{\rho}(\sigma)$ is holomorphic at $\sigma=0$ so that 
	\begin{align}
		\left|R_{\rho}(0)\right|=T(F)^{2}. 
	\end{align} 	
\end{cor}

In Section \ref{sPC2}, we will show that the  unitarily flat vector 
bundle is  contained in the class specified above. Indeed, let  
$\rho_{0}:\Gamma\to {\rm U}(r)$ be the holonomy representation of a unitarily flat vector 
bundle $(F_{0},g^{F_{0}})$. If $\ul{G}=G\times {\rm U}(r)$, $\ul{K}=K\times {\rm 
U}(r)$, and if $\ul{\Gamma}\subset \Gamma\times \mathrm U(r)$ is the 
graph of $\rho_{0}$, then we have the identification 
$Z=\ul{\Gamma}\backslash \ul{G}/\ul{K}$. If $\ul{\rho}_{0}:\ul{G}\to 
{\rm U}(r)$ is the projection onto the second component, 
it is easy to see that $\ul{\rho}_{0}$ has an admissible metric and 
the associated Hermitian flat vector bundle is just  
$(F_{0},g^{F_{0}})$. 
In this way, we show  : 

\begin{cor}\label{cori2}
	Assume that $\dim Z$ is odd. If $(F_{0},g^{F_{0}})$ is a 
	unitarily flat vector bundle, and if $(F,g^{F})$ is the Hermitian 
	flat vector bundle as in Theorem \ref{Thm1}, then the statements of 
	Theorem \ref{Thm1} and Corollary \ref{cori1} hold for  $F_{0}\otimes F$. 
\end{cor} 

\begin{re}
The flat vector bundle $F_{0}\otimes F$ in Corollary \ref{cori2} is of  particular interest in study of hyperbolic volumes (see e.g. \cite{BDHP2019}).
\end{re}

In Section \ref{Snontorsionfree}, we will extend all the above results to the case where $\Gamma$ is not torsion free. Then $Z$ is an orbifold and $F$ is a flat orbifold vector bundle. 

\begin{thm}\label{Thm2}
	The statement of Theorem \ref{Thm1} and Corollaries \ref{cori1} 
	and 	\ref{cori2} hold for orbifolds.  
\end{thm}

\subsection{Proof of Theorem \ref{Thm1}}
We will first show Theorem \ref{Thm1} in the case $\rho^\theta\simeq 
\rho$, i.e., Theorem \ref{thm1}. 
Using an easy relation  
$R_{\rho^{\theta}}(\sigma)=\ol{R_{\rho}(\ol{\sigma})}$ (Proposition 
\ref{propRoR}) and  applying Theorem \ref{thm1} to  
$\rho\oplus \rho^{\theta}$, we obtain Corollary \ref{cori1}. Since 
$\rho$ has an admissible metric, $\rho$ can be decomposed as  a 
direct sum of irreducible representations which are either 
$\rho^{\theta}\simeq \rho$ or $\rho^{\theta}\not\simeq \rho$. In this 
way, we get Theorem \ref{Thm1} in full generality. 

Our proof of Theorem \ref{Thm1} in the case $\rho^\theta\simeq 
\rho$ is inspired by 
\cite{Shfried} and \cite{Muller_3_torsion}. Let us explain  
main steps. 

\subsubsection{Moscovici-Stanton's vanishing theorem}
Let $\delta(G)\in \mathbf N$ be the fundamental rank of $G$, i.e., 
the difference between the complex ranks of $G$ and $K$. Note that $\delta(G)$ and $\dim Z$ have the same parity.

When $F$ is unitary, by \cite[Corollary 2.2, Remark 3.7]{MStorsion}, 
if $\delta(G)\g 3$, we have 
\begin{align}
&T(F)=1,&	R_{\rho}(\sigma)\equiv 1. 
\end{align} 
In the current situation, similar results still hold (e.g. \cite{BMZ1}
\cite[Theorem 8.6, Remark 8.7]{BMZ}, \cite[Theorem 5.5]{Ma_bourbaki}, \cite[Remark 4.2]{Shen20}). Therefore, we can reduce the proof to the case $\delta(G)=1$.


\subsubsection{Selberg zeta functions}
%
Assume now $\delta(G)=1$ and $\rho^{\theta}\simeq \rho$.  The proof of Theorem \ref{Thm1} in this case is based on the introduction of the Selberg zeta functions. Let us recall its definition and basic properties. 

Let $\ft\subset \fk$ be a Cartan subalgebra of $\fk$. Let 
$\fh\subset \fg$ be the stabiliser of $\ft$ in $\fg$. By 
\cite[p.~129]{Knappsemi}, 
$\fh\subset \fg$ is a  $\theta$-invariant fundamental Cartan 
subalgebra of $\fg$. Let 
$\fh=\fb\oplus \ft$ be the Cartan decomposition of $\fh$. Note that  
$\dim \fb=\delta(G)=1$.  Let $H\subset G$ be the associated Cartan 
subgroup of $G$.  

Let $Z(\fb)\subset G$ be the 
stabiliser of $\fb$ in $G$ with Lie algebra $\fz(\fb)$. Let $Z^{0}(\fb)$ be the connected 
component of the identity in $Z(\fb)$.  Then $\fz(\fb),Z^{0}(\fb)$ split
\begin{align}
&\fz(\fb)=\fb\oplus\fm, &Z^{0}(\fb)=\exp(\fb)\times M,
\end{align} 
where $M$ is a connected reductive subgroup of $G$ with Lie algebra 
$\fm$. Let $\fm=\fp_{\fm}\oplus \fk_{\fm}$ be the Cartan 
decomposition of $\fm$.  Let 
$\fz^{\bot}(\fb)\subset \fg$ be the orthogonal space of 
$\fz^{\bot}(\fb)$ with respect to  $B$. 


Let $\eta=\eta^{+}-\eta^{-}$ be a virtual representation of $M$ 
acting on the finite dimensional complex vector spaces $E_{\eta}=E_{\eta}^{+}-E_{\eta}^{-}$ 
such that 
	\begin{enumerate}[i)]
	\item\label{ini}  the Casimir of $M$ acts on $\eta^{\pm}$ by the same 
	scalar;
	\item\label{inii}  the restriction of $\eta$ to $K_{M}=K\cap M$ lifts uniquely to a virtual 
representation of $K$.
	\end{enumerate} 
The Selberg zeta function associated to $\eta$ is defined 	formally for $\sigma\in \bC$ by 
	\begin{align}\label{ieqdefsel}
		Z_{\eta}(\sigma)=\exp\Bigg(-\sum_{\tiny\substack{[\gamma]\in 
		[\Gamma_{+}]\\ \gamma\sim e^{a}k^{-1}\in H}}\frac{\chi_{\rm 
		orb}\(B_{[\gamma]}/\mathbb{S}^{1}\)}{m_{[\gamma]}}\frac{\Trs^{E_{\eta}}[k^{-1}]}{\left|\det\(1-\Ad(e^{a}k^{-1})\)|_{\fz^{\bot}(\fb)}\right|^{1/2}}e^{-\sigma \ell_{[\gamma]}}\Bigg),
\end{align} 
where the sum is taken  over the non elliptic conjugacy classes 
$[\gamma]$ of $\Gamma$ such that $\gamma$ can be conjugate by element 
of $G$ into the  Cartan subgroup $H$. 

In \cite[Section 6]{Shfried} and \cite[Section 3.4]{Shen20}, we have shown that the adjoint action of $K_{M}$ on $\fp_{\fm,\bC}$ lifts uniquely to 
a  virtual representation  of $K$. Let 
$\widehat{\eta}=\widehat{\eta}^{+}-\widehat{\eta}^{-}$ be the unique virtual  representation of $K$ such that 
\begin{align}
\widehat{\eta}|_{K_{M}}=	\Lambda^{\scriptscriptstyle\bullet 
}(\fp_{\fm,\bC}^{*})\widehat{\otimes} \eta_{|K_{M}}. 
\end{align} 
The Casimir operator of $\fg$ acts as a generalised Laplacian
 $C^{\fg,Z,\widehat{\eta}^{\pm}}$ 
on the smooth sections over $Z$ of the locally homogenous vector bundle induced by 
$\widehat{\eta}^{\pm}$ (see \eqref{eqFtau0}).  By the  general theory on elliptic differential  operators, the regularised determinant
${\rm det} 
\big(C^{\fg,Z,\widehat{\eta}^{\pm}}+\sigma\big)$ is holomorphic on $\sigma\in \bC$. 

In \cite[Section 7]{Shfried} and \cite[Section 5]{Shen20}, we show that $Z_{\eta}(\sigma)$ has a meromorphic extension to $\sigma \in \bC$. 
Moreover, up to a multiplication by a non zero entire function, $Z_{\eta}(\sigma)$ is just the graded regularised determinant 
\begin{align}\label{eqdetCi1}
\frac{	{\rm det} 
\big(C^{\fg,Z,\widehat{\eta}^{+}}+\sigma_{\eta}+\sigma^{2}\big)}{{\rm det} 
\big(C^{\fg,Z,\widehat{\eta}^{-}}+\sigma_{\eta}+\sigma^{2}\big)},
\end{align} 
where $\sigma_{\eta}\in \bR$ is some constant.

One of main steps in our proof of Theorem \ref{Thm1} is to construct  a  family of 
virtual $M$-representations $\eta_{\beta}$ satisfying the above 
assumptions \ref{ini}) \ref{inii}), parametrised by finite elements 
$\beta\in \fb^{*}$,  such that 
\begin{align}\label{eqAS333in}
	\frac{	\sum_{\beta\in 
		\fb^{*}}e^{\<\beta,a\>}\Tr_{\rm 
		s}\[\eta_{\beta}\(k^{-1}\)\]}{
		\left|\det\(1-\Ad(e^{a}k^{-1})\)|_{\fz^{\bot}(\fb)}\right|^{1/2}}=\Tr\[\rho\(e^{a}k^{-1}\)\], 
	\end{align}
and such that the following identity of virtual $K$-representations holds, 
	\begin{align}\label{eqsumbetainto}
		\bigoplus_{\beta\in 
		\fb^{*}}\widehat{\eta}_{\beta}=\bigoplus_{i=1}^{m}(-1)^{i-1}i 
		\Lambda^{i}\(\fp^{*}_{\bC}\)\otimes \rho_{|K}. 
	\end{align}
Using \eqref{eqAS333in}, we 	can write $R_{\rho}$ as  an 
	alternating product of $Z_{\eta_{\beta}}$. Thanks to 
	\eqref{eqdetCi1}, we get a relation between the Ruelle zeta 	
	function and the Casimir operator, which is the Hodge Laplacian 
	of $(g^{TZ},g^{F})$ by \eqref{eqsumbetainto} (c.f. Proposition 
	\ref{prop:D=C-C}). In this way, we get \eqref{eq1rdii}. 
		
\subsubsection{Dirac cohomology}
The construction of $\eta_{\beta}$ is based on the Dirac cohomology  \cite{HuangDiracCoh}.  Recall that  in \cite{Shfried}, we have shown that $(\fg,\fz(\fb))$ is a symmetric 
pair.  Let $(\fu,\fu(\fb))$ be the associated compact symmetric pair.  The Dirac cohomology $H_{D}^{\pm}(\rho)$ of the $\fu$-representation $\rho$ with 
respect to the symmetric pair $(\fu,\fu(\fb))$ is a 
$(\fb_{\bC}\oplus \fm_{\bC},K_{M})$-modules. We define 
$\eta_{\beta}^{\pm}$ to be the $(\fm_{\bC},K_{M})$-modules such that 
\begin{align}
H^{\pm}_{D}(\rho)\simeq \bigoplus_{\beta\in \fb^{*}} 
\bC_{\beta}\boxtimes \eta_{\beta}^{\pm},
\end{align}
where $\bC_{\beta}$ is the one dimensional representation of $\fb$ such that $a\in 
\fb$ acts as $\<\beta,a\>\in \bR$.  The virtual $M$-representation of $\eta_{\beta}$ is  defined by 
$\eta_{\beta}^{+}-\eta_{\beta}^{-}$. Now, the assumption 
\ref{ini})\footnote{More precisely, we need  assume that the Casimir of $\fg$ acts on $\rho$ as a scalar.}, 
\eqref{eqAS333in}, and 
\eqref{eqsumbetainto} are easy consequences of properties  of the Dirac cohomology.  In Section 
\ref{srepetab}, we show the assumption \ref{inii}) as well, so that 
the Selberg zeta function of $\eta_{\beta}$ is well defined. 

%
%
%
%
%
%
%
%


Let us remark that the Dirac cohomology is more or less equivalent 
to  the $\fn$-cohomology used in \cite{Muller_3_torsion}. However, Dirac cohomology is closer to the spin construction used in \cite[Section 6]{Shfried}.

\subsubsection{Infinitesimal 
character and the vanishing of $(\fg,K)$-cohomology}The proof of \eqref{eq2rdin} is 
based on a
relation between the infinitesimal 
character and the  vanishing of $(\fg,K)$-cohomology of a unitary 
Harish-Chandra $(\fg_{\bC},K)$-module, which is due to 
Vogan-Zuckermann \cite{VoganZuckerman}, Vogan \cite{Vogan2}, and 
Salamanca-Riba \cite{Salamanca}. The idea of the proof is very 
similar to the one given in \cite[Section 8]{Shfried}. We refer the reader to \cite[Section 1H]{Shfried} for an introduction. 

\subsubsection{Final  remark}
In the case where $G$ have a noncompact 	 centre, we have $\fn=0$. 
It is  unnecessary to use  the Dirac cohomology and  the results Vogan, Zuckermann, 
and Salamanca-Riba.  For greater clarity, we single 
	out this case in Section \ref{sproofth1cc}.

%

%
%
%
%
%
%
%
%
%

\subsection{The organisation of the article}
This article is organised as follows.
In Section \ref{SgLap}, we recall the definition of the Ray-Singer 
analytic torsion of a flat vector bundle. 

In Section \ref{Sreductive}, we introduce the real reductive group $G$ 
and the admissible metrics on finite dimension representations of 
$G$. 

In Section \ref{Sselberg}, we recall the definition and the 
proprieties of the zeta functions of Ruelle and Selberg established 
in \cite{Shfried,Shen20}. 

In Section \ref{SFCAM}, we state our main Theorem \ref{thm1}, from 
which we  deduce Theorem  \ref{Thm1}, Corollaries \ref{cori1} and 
\ref{cori2}. Also, we prove  Theorem \ref{thm1} when $\delta(G)\neq 1$ or $Z_{G}$ is non  compact.  

In Sections  \ref{SPdn1} and \ref{S:rep},  we establish Theorem \ref{thm1} when $\delta(G)=1$ and $Z_{G}$ is compact. 

Finally, in Section \ref{Snontorsionfree}, we extends the previous results to  orbifolds and we show Theorem \ref{Thm2}.

\subsection{Notation}
Throughout the paper, we use the superconnection formalism of 
\cite{Quillensuper} and \cite[Section 1.3]{BGV}. If $A$ is a 
$\mathbf{Z}_2$-graded algebra and if $a, b\in A$, the 
supercommutator $[a,b]$ is given by
\begin{align}
ab-(-1)^{\deg a \deg b}ba.
\end{align}
If $B$ is another $\mathbf{Z}_2$-graded algebra, we denote 
by $A  \widehat{\otimes}  B$ the super tensor product algebra of $A$ and $B$. 
If $E=E^{+}\oplus E^{-}$ is a $\mathbf{Z}_2$-graded vector 
space, the algebra $\End(E)$ is $\mathbf{Z}_2$-graded. If 
$\tau=\pm1$ on $E^{\pm}$ and if $a\in \End(E)$, the supertrace 
$\Trs[a]$ is defined by $\Tr[\tau a]$. 

If $M$ is a topological group, we will denote by $M^{0}$ the 
connected component of the identity in $M$.  If $V$ is a real vector 
space, we will use the notation $V_{\bC}=V\otimes_{\bR}\bC$ for its complexification.  
We make the convention that $\mathbf{N}=\{0,1,2,\ldots\}$, 
$\mathbf{N}^{*}=\{1,2,\ldots\}$, $\bR_{+}^{*}=(0,\infty)$.

\settocdepth{subsection}

\section{The analytic torsion} \label{SgLap}
Let $Z$ be a closed smooth manifold of dimension $m$. Let $(F,\nabla^{F})$ be a flat complex vector bundle on $Z$ with flat connection $\nabla^{F}$. 
Let $(\Omega^{\scriptscriptstyle\bullet}(Z,F),d^{Z})$ be the de 
Rham complex of smooth sections of 
$\Lambda^{\scriptscriptstyle\bullet }(T^{*}Z)\otimes_{\bR} F$ on $Z$. Let $H^{\scriptscriptstyle\bullet}(Z,F)$ be the   de Rham cohomology.  

We define the Euler characteristic  number $\chi(Z,F)$ and the 
derived Euler characteristic  number $\chi'(Z,F)$ by 
\begin{align}\label{eqEuler}
 &	\chi(Z,F)=\sum_{i=0}^{m}(-1)^{i}\dim 
 H^{i}(Z,F),&\chi'(Z,F)=\sum_{i=1}^{m}(-1)^{i}i\dim H^{i}(Z,F). 
 \end{align}
 If $F$ is trivial, we write $\chi(Z)$ and $\chi'(Z)$. 

Let $g^{TZ}$ be a Riemannian metric on $Z$. Let $g^{F}$ be a 
Hermitian metric on $F$. The metrics $g^{TZ},g^{F}$ induce a scalar 
product 
$\<,\>_{\Lambda^{\scriptscriptstyle\bullet}(T^{*}Z)\otimes_{\bR} F}$
on $\Lambda^{\scriptscriptstyle\bullet}(T^{*}Z)\otimes_{\bR} F$. 
Let $\<,\>_{L^{2}}$ be an $L^{2}$-product  on 
$\Omega^{\scriptscriptstyle\bullet}(Z,F)$ defined for $s_{1},s_{2}\in 
\Omega^{\scriptscriptstyle\bullet}(Z,F)$ by 
	\begin{align}\label{eqL2prod}
		\<s_{1},s_{2}\>_{L^{2}}=\int_{z\in 
		Z}\<s_{1}(z),s_{2}(z)\>_{\Lambda^{\scriptscriptstyle\bullet 
		}(T^{*}Z)\otimes_{\bR} F}dv_{Z},
	\end{align} 
where $dv_{Z}$ is the Riemannian volume form of $(Z,g^{TZ})$. 	
	
Let $d^{Z,*}$ be the formal 
adjoint of $d^{Z}$. Set 
\begin{align}
 \Box^Z=\[d^Z,d^{Z,*}\].
\end{align}
Then, $\Box^{Z}$ is a second order self-adjoint elliptic differential operator acting on $\Omega^{\scriptscriptstyle\bullet}(Z,F)$. 
By Hodge theory, we have 
\begin{align}\label{eqHo1}
\ker \Box^{Z}=H^{\scriptscriptstyle\bullet }(Z,F). 
\end{align} 

Let $\(\ker \Box^{Z}\)^{\bot}$ be the orthogonal vector space to 
$\ker \Box^{Z}$  in $\Omega^{\scriptscriptstyle\bullet }(X,F)$.
Then $\Box^{Z}$ acts as an invertible operator on $\(\ker 
\Box^{Z}\)^{\bot}$. Let $\(\Box^{Z}\)^{-1}$ denote the inverse of $\Box^{Z}$ acting on $\(\ker 
\Box^{Z}\)^{\bot}$.
If $0\l i \l m$, for $s\in \bC$ and $\Re(s)>\frac{m}{2}$, set 
\begin{align}
	\theta_{i}(s)=\Trs\[\(\Box^{Z}\)^{-s}_{|\Omega^{i}(Z,F)}\]. 
\end{align} 
By \cite{Seeley66} and \cite[Proposition 9.35]{BGV}, $\theta_{i}(s)$ extends to a meromorphic function of $s\in \bC$, 
which is holomorphic at $s=0$. The regularised determinant is defined by 
\begin{align}
	{\rm det}^{*}\(\Box^{Z}_{|\Omega^{i}(Z,F)}\) =\exp 
	\(-\theta_{i}'(0)\). 
\end{align} 
Formally, it is the product of non zero eigenvalues counted  with multiplicities. 

\begin{defin}
The Ray-Singer analytic torsion \cite{RSTorsion} of $F$ is defined by
\begin{align}\label{eqRST}
T(F)=	\prod_{i=1}^{m}{\rm 
det}^{*}\(\Box^{Z}_{|\Omega^{i}(Z,F)}\)^{(-1)^{i}i/2}\in \bR_{+}^{*}.
\end{align}
\end{defin}

By \cite{RSTorsion} and \cite[Theorem 0.1]{BZ92},  if $\dim Z$ is odd and if $H^{\scriptscriptstyle\bullet 
}(Z,F)=0$, then $T(F)$ does not depend on the metrics $g^{TZ},g^{F}$. It becomes a topological invariant.  

For  $0\l i\l m$, if $\sigma>0$, the operator $\sigma+\Box^{Z}_{|\Omega^{i}(Z,F)}$ does not contain the 
zero spectrum, we denote its regularised determinant by  $	\det\(\sigma+\Box^{Z}_{|\Omega^{i}(Z,F)}\).$
By \cite{Voros} and \cite[Theorem 1.5]{Shen20}, the function $ 
\det\(\sigma+\Box^{Z}_{|\Omega^{i}(Z,F)}\)$ 
extends to  a holomorphic function of  $\sigma\in \bC$, whose zeros 
are located  at $\sigma=-\lambda$ with  order $\dim \ker\( 
\Box^{Z}_{|\Omega^{i}(Z,F)}-\lambda\)$, where  $\lambda\in 
\Sp\(\Box^{Z}_{|\Omega^{i}(Z,F)}\)$.


Set
\begin{align}\label{eqTsigma}
	T(\sigma)=	\prod_{i=1}^{m}{\rm 
det}\(\sigma+\Box^{Z}_{|\Omega^{i}(Z,F)}\)^{(-1)^{i}i}.
\end{align} 
Then, $T(\sigma)$ is  meromorphic.  By \eqref{eqEuler},  
\eqref{eqRST}, and \eqref{eqTsigma},  as $\sigma\to 0$, we have 
\begin{align}
	T(\sigma)=T(F)^{2}\sigma^{\chi'(Z,F)}+\cO(\sigma^{\chi'(Z,F)+1}). 
\end{align}

\section{Reductive groups and  finite dimensional representations}\label{Sreductive}
The purpose of this section is to recall some basic facts about  real reductive groups and their finite dimensional representations. 

This section is organised as follows. 
In Sections \ref{sReductive}-\ref{sSemi}, we introduce the real 
reductive  group $G$, its Lie algebra $\fg$, the enveloping algebra 
$\mathscr U(\fg)$, the Casimir operator, the Dirac operator, as well as the semisimple elements. 

In Sections \ref{sCartanF} and \ref{sSpliting}, we introduce the 
fundamental Cartan subalgebra of $\fg$ and some related 
constructions. We recall a key lifting proprieties established in 
\cite[Section 3.4]{Shen20}. 

In Section \ref{sadRepG},  we recall the definition of admissible 
metrics on   finite dimensional representations of $G$ introduced by \cite{MatsuchimaMurakami}. We 
show that if $G$ has a compact centre, then all the finite dimensional representations have  admissible metrics.

\subsection{Real reductive groups}\label{sReductive}
Let $G$ be a linear connected real reductive group \cite[p.~3]{Knappsemi}, and let $\theta \in {\rm Aut}(G)$ be the Cartan 
involution. That means  $G$ is a closed connected group of real matrices that is stable under transpose, and $\theta$ is the composition of transpose 
and inverse of matrices.  If $\fg$ is the Lie algebra of $G$, then  $\theta$ acts as an automorphism on $\fg$. 

Let $K \subset G$ be the fixed point set of 
$\theta$ in $G$. Then $K$ is a compact connected subgroup of $G$, which is a 
maximal compact subgroup. 
%
If $\fk$ is the Lie algebra of $K$, then  $\fk$ is the eigenspace of  
$\theta$ associated with the eigenvalue $1$. 
Let $\fp$ be the eigenspace of $\theta$ associated with the 
eigenvalue $-1$, so that  we have the Cartan decomposition
\begin{align}\label{eq:Cartande}
\fg = \fp \oplus \fk.
\end{align}
Set
\begin{align}
&	m=\dim \fp,&n=\dim \fk.
\end{align}
By \cite[Proposition 1.2]{Knappsemi}, we have the diffeomorphism
\begin{align}\label{eq:cartan2}
 (Y,k)\in \fp\times K\to  e^Y k\in G.
\end{align}

Let $B$ \index{B@$B$} be a  nondegenerate bilinear real  
symmetric form on $\fg$ which is invariant under the adjoint action 
$\Ad$ of $G$, and also under $\theta$. Then 
\eqref{eq:Cartande} is an orthogonal splitting of $\fg$ with respect 
to $B$. We assume $B$ to be positive-definite on $\fp$, and negative-definite on $\fk$. 
Then,  $\<\cdot,\cdot\>=-B(\cdot,\theta\cdot)$ defines an 
$\Ad(K)$-invariant scalar 
product on $\fg$ such that the splitting 
\eqref{eq:Cartande} is still orthogonal. We denote by $|\cdot|$ \index{1@$\lvert\cdot\rvert$}the corresponding norm.

Let $Z_G\subset G$\index{Z@$Z_G$} be the centre of $G$ with Lie algebra $\fz_\fg\subset \fg$.\index{Z@$\fz_\fg$}
By \cite[Corollary 1.3]{Knappsemi}, $Z_G$ is a (possibly non 
connected) reductive group  with 
maximal compact subgroup $Z_{G}\cap K$ with the Cartan decomposition 
\begin{align}
\fz_\fg=\fz_{\fp}\oplus \fz_\fk. 	
\end{align} 
 Since $\fz_{\fp}$ commutes with $Z_{G}\cap K$, by \eqref{eq:cartan2}, 
we have an identification  of the groups
\begin{align}
	 Z_G=\exp (\fz_\fp) \times (Z_G\cap K).
\end{align}

Let $\fg_\bC=\fg\otimes_\bR\bC$ \index{G@$\fg_\bC$}be the complexification of $\fg$ and let $\fu=\sqrt{-1}\fp\oplus \fk$\index{U@$\fu$}
be the compact form of $\fg$. By $\bC$-linearity, the bilinear form $B$ extends 
to a complex symmetric bilinear form on $\fg_{\bC}$. The restriction 
$B|_{\fu}$ to $\fu$ is real and negative-definite. 

Let $\mathscr{U}(\fg)$ and $\mathscr U(\fg_{\bC})$ be the enveloping 
algebras 
of $\fg$ and $\fg_{\bC}$.  Let $\mathscr{Z}(\fg)$ 
\index{Z@$\mathcal{Z}(\fg)$} and $\mathscr Z(\fg_{\bC})$ be respectively the 
centres of $\mathscr{U}(\fg)$ and $\mathscr U(\fg_{\bC})$. Clearly, 
\begin{align}\label{eqZcZc}
&	\mathscr U(\fg_{\bC})=\mathscr U(\fg)\otimes_{\bR} \bC,&\mathscr 
Z(\fg_{\bC})=\mathscr Z(\fg)\otimes_{\bR} \mathbf{C}. 
\end{align}

If $V$ is a complex vector space, and if $\rho:\fg\to 
\End(V)$ is a representation of $\fg$, then the map $\rho$ extends to 
a  morphism $\rho:\mathscr{U} (\fg_{\bC}) \to \End (V)$ of algebras. 

If $\rho : G
\to \GL(V)$ is a finite dimensional complex representation of $G$, then the induced   morphism  $\rho:\mathscr{U} (\fg_{\bC}) \to \End (V)$ of algebras is 
$K$-equivalent.  In this way, $V$ is a finite dimensional $(\fg_\bC,K)$-module, i.e., a  
$\mathscr{U}(\fg_\bC)$-module, equipped 
with a compatible $K$-action. By \cite[Proposition 4.46]{KnappCohomology}, it is equivalent to 
consider finite dimensional representations of $G$ and finite 
dimensional  $(\fg_\bC,K)$-modules. In the sequel, we will not 
distinguish these two objets. 

\begin{re}\label{regGsub}
	If $\rho:G\to \GL(V)$ is a finite dimensional complex 
	representation of $G$, and if 	$W\subset V$ is a $\fg$-invariant 
	subspace, by taking the exponential of the action of $\fg$, we 
	see that the 	
	group $G$ preserves $W$. In particular, 
	the  set of $\fg$-subrepresentations of $\rho$ coincides with the set of 	$G$-subrepresentations of  $\rho$.
\end{re}



\subsection{The Casimir operator}\label{sCasimir}
Let $C^\fg\in \mathscr{Z}(\fg)$ be the Casimir element associated to 
$B$. If $e_1,\cdots,e_m$ is an orthonormal basis of $(\fp,B|_{\fp})$, 
and if $e_{m+1},\cdots,e_{m+n}$ is an  orthonormal basis of 
$(\fk,-B|_{\fk})$, then 
\begin{align}\label{eq:Cg}
  C^\fg=-\sum_{i=1}^{m}e^2_i+\sum_{i=m+1}^{n+m}e_i^{2}.\index{C@$C^{\fg}$}
\end{align}

%

If $\rho:\fg\to \End( V)$ is a complex representation of $\fg$, we  denote by $C^{\fg,V}$ or $C^{\fg,\rho}\in \End(V)$ the corresponding 
Casimir operator acting on $V$, i.e.,
\begin{align}\label{eq:ckkckp}
	C^{\fg,V}=C^{\fg,\rho}=\rho(C^{\fg}).
\end{align}
Similarly,  the Casimir of $\fu$ (with respect to $B$)  acts on 
$V$, so that 
\begin{align}\label{eqCu=Cg}
	C^{\fu,V}=C^{\fg,V}. 
\end{align}

\subsection{The Dirac operator}\label{sDirac}
Let $c(\fp)$ be the Clifford algebra of $(\fp,B|_{\fp})$. That is an 
algebra over $\bR$ generated by $1\in \bR$, $a\in \fp$ with the commutation relation 
for  $a_1,a_{2}\in \fp$, 
\begin{align}
	a_{1}a_{2}+a_{2}a_{1}=-2B(a_{1},a_{2}). 
\end{align}

Let $S^{\fp}$ be the spinor of $(\fp,B|_{\fp})$. If $a\in 
\fp$, the action of $a$ on $S^{\fp}$ is denoted by  $c(a)$. If $a\in 
\fk$, $\ad(a)|_{\fp}$ acts as an antisymmetric endomorphism on $\fp$. 
It acts on $S^\fp$  by
\begin{align}\label{eqcada}
	c\(\ad(a)|_{\fp}\)=\frac{1}{4}\sum_{i,j=1}^{m}\<[a,e_{i}],e_{j}\>c(e_{i})c(e_{j}).
\end{align}


Let $\rho:\fg\to \End(V)$ be a complex representation of $\fg$.   Let $D^{S^{\fp}\otimes 
	V}$ be the Dirac operator acting on $S^{\fp}\otimes 
	V$, i.e., 
\begin{align}\label{eqcDirac}
	D^{S^{\fp}\otimes 
	V}=\sum_{i=1}^{m}c(e_{i})\rho(e_{i}).
\end{align}
Recall that $\fk$ acts on $\fp$ by adjoint action. The operator 
$C^{\fk,\fp}$ is defined in \eqref{eq:ckkckp}. 

\begin{prop}
The following identity of operators on $S^{\fp}\otimes V$ holds, 
\begin{align}\label{eqD2}
		\(D^{S^{\fp}\otimes 
	V}\)^{2}=C^{\fg,V}+\frac{1}{8}\Tr\[C^{\fk,\fp}\]-C^{\fk,S^{\fp}\otimes 
	V}.
	\end{align}
\end{prop}
\begin{proof}	
	This is a consequence of  \cite[Lemma 	II.6.11]{BW} and of 
	Kostant's strange formula \cite{Kostant76} (see also 
	\cite[(2.6.11), (7.5.4)]{B09}). 
	
	There is another way proving this result, by imitating the proof of \cite[Theorem 
	7.2.1]{B09} which uses  Kostant's cubic Dirac operator \cite{Kostant76,Kostant97}. 
\end{proof}

In the sequel,  we will also consider the Dirac operator associated to the compact 
	symmetric pair $(\fu,\fk)$ with respect to the positive bilinear 
	form $-B|_{\fu}$. Using the identification $a\in \fp\to \sqrt{-1}a\in  \sqrt{-1}\fp$, 
we can identify  $c(\fp)$ with the 
Clifford algebra of  the Euclidean space $(\sqrt{-1}\fp,-B)$. Also,  
$S^{\fp}$ can be identified  with   the  
spinor $S^{\sqrt{-1}\fp}$ of $(\sqrt{-1}\fp,-B)$.  Then, the Dirac operator 
$D^{S^{\sqrt{-1}\fp}\otimes 
		V}$ is just  $\sqrt{-1}D^{S^{\fp}\otimes 
		V}$. By \eqref{eqCu=Cg} and \eqref{eqD2}, we have  
\begin{align}\label{eqD2cp}
		-\(D^{S^{\sqrt{-1}\fp}\otimes 
		V}\)^{2}=C^{\fu,V}+\frac{1}{8}\Tr\[C^{\fk,\sqrt{-1}\fp}\]-C^{\fk,S^{\sqrt{-1}\fp}\otimes 
	V}.
	\end{align}

\subsection{Semisimple elements}\label{sSemi}If $\gamma\in G$, we denote by $Z(\gamma) \subset G$ the centraliser 
of $\gamma$ in $G$, and by $\fz(\gamma)\subset \fg$ its Lie algebra. 
If $a\in \fg$, let $Z(a)\subset G$ be the stabiliser of $a$ in $G$, 
and let $\fz(a)\subset \fg$ be its Lie algebra. If $\fa \subset \fg$  is a subset, we define $Z(\fa)$ and $\fz(\fa)$ similarly.

Let $\gamma\in G$ be a semisimple element, i.e.,  there is $g_{\gamma}\in G$ 
 such that  $\gamma=g_{\gamma}e	
 ^{a}k^{-1}g_{\gamma}^{-1}$ and
 \begin{align}\label{eqasr}
 	&a\in \fp,&k\in K,&& \Ad(k)a=a.  
 \end{align}
The norm $|a|$ depends only on the conjugacy class of $\gamma$ 
 in $G$. Write 
 \begin{align}
 	\ell_{[\gamma]}=|a|. 
 \end{align}
 A semisimple element $\gamma$ is  called elliptic, if $\ell_{[\gamma]}=0$. 

If $\gamma$ is semisimple, by \cite[Proposition 
7.25]{KnappLie}, $Z(\gamma)$ is a (possibly non connected) reductive 
group with Cartan involution\footnote{By \cite[Theorem 2.3]{BS19}, this is indeed independent of 
the choice of $g_{\gamma}$.} $g_{\gamma}\theta 
g_{\gamma}^{-1}$. Let  $K(\gamma)\subset Z(\gamma)$ be the associated maximal compact subgroup of $Z(\gamma)$.  



\subsection{The fundamental Cartan subalgebra}\label{sCartanF}
Let $T\subset K$ be a maximal torus of $K$. Let $\ft\subset \fk$ be the Lie algebra of $T$. If $N_{K}(T)$ 
is the normaliser of  $T$ in $K$, let $W(T:K)=N_{K}(T)/T$ be associated  Weyl group.  

Set
\begin{align}\label{eqh=bt}
	&\fb=\left\{a\in \fp: [a,\ft]=0\right\}, &\fh=\fb\oplus \ft.
\end{align} 
By \cite[p.~129]{Knappsemi}, $\fh$ is a Cartan subalgebra of $\fg$.  Let 
$H=Z(\fh)$ be the associated Cartan subgroup of $G$. By \cite[Theorem 
5.22]{Knappsemi}, $H$ is a connected abelian reductive subgroup of $G$, so that 
\begin{align}
	H=\exp(\fb)\times T. 
\end{align} 
We will call $\fh$ and $H$ respectively the fundamental Cartan subalgebra of $\fg$  and the fundamental Cartan subgroup of $G$.

Recall that complex ranks of $G$ and $K$ are defined respectively by 
the dimensions  of  Cartan subalgebras of $\fg_{\bC}$ and $\fk_{\bC}$. 

\begin{defin}
  The fundamental rank $\delta(G)$ of $G$ is defined by the difference of complex 
  ranks of $G$ and $K$, i.e., 
  \begin{align}\label{eqd=g-k}
	\delta(G)=\dim \fb. 
	\end{align}
	Note  that $m$ and $\delta(G)$ have the same parity. 
\end{defin}

\subsection{A splitting of $\fg$ according to the 
$\fb$-action}\label{sSpliting}
By \cite[Proposition 
7.25]{KnappLie},  $Z(\fb)$ is a (possibly non connected)  reductive 
subgroup of  $G$,  so that we have the Cartan 
 decomposition
 \begin{align}\label{eqzb0}
 	\fz(\fb)=\fp(\fb)\oplus \fk(\fb). 
 \end{align}

Let $\fm\subset 
\fz(\fb)$ be the orthogonal  space (with respect to B) of $\fb$ in $\fz(\fb)$. Then $\fm$ 
is a Lie subalgebra of $\fg$, and  $\theta$ acts on $\fm$ so that we have the Cartan 
 decomposition
 \begin{align}\label{eqmb0}
 	\fm=\fp_{\fm}\oplus \fk_{\fm}. 
 \end{align}
Let 
$M\subset G$ be the connected Lie subgroup associated to the  Lie 
algebra $\fm$. By \cite[(3.3.11) and Theorem 3.3.1]{B09}, $M$ is  
closed in $G$ and is  a connected reductive subgroup of $G$ 
with maximal compact subgroup 
\begin{align}\label{eqKM0}
	K_{M}=M\cap K.
\end{align}
Moreover, we have
\begin{align}\label{eqZ(b)}
&Z^{0}(\fb)=\exp(\fb)\times M,&\fz(\fb)=\fb\oplus \fm, &&	
\fp(\fb)=\fb\oplus\fp_{\fm},&&& \fk(\fb)=\fk_{\fm}. 
\end{align}
Since $\fh$ is also a Cartan subalgebra for $\fz(\fb)$, we have
\begin{align}\label{eqdm=0}
	\delta(M)=0. 
\end{align}

Let $\fp^{\bot}(\fb),\fk^{\bot}(\fb),\fz^{\bot}(\fb)$ be respectively the orthogonal  spaces (with respect to $B$) of $\fp(\fb),\fk(\fb),\fz(\fb)$ 
in $\fp,\fk,\fg$. Clearly,
\begin{align}
	\fz^{\bot}(\fb)=\fp^{\bot}(\fb)\oplus \fk^{\bot}(\fb). 
\end{align}
And also 
\begin{align}\label{eq:mpk1}
 & \fp=\fb\oplus\fp_\fm\oplus\fp^\bot(\fb),&\fk=\fk_\fm\oplus\fk^\bot(\fb),&& \fg=\fb\oplus \fm\oplus \fz^{\bot}(\fb).
\end{align}
The group $K_M$ acts trivially on $\fb$. It also acts on $\fp_\fm$, 
$\fp^{\bot}(\fb)$, $\fk_\fm$ and $\fk^{\bot}(\fb)$, and preserves the 
splittings \eqref{eq:mpk1}. Similarly, the groups $M$ and $Z^0(\fb)$ 
act trivially on $\fb$, act on $\fm,\fz^{\bot}(\fb)$, and preserves the third splitting in \eqref{eq:mpk1}.

\begin{re}
	We can define similar objects associated to the action of 
$\sqrt{-1}\fb\subset \fu$ on $\fu$.  Let $\fu_{\fm}$ and $\fu(\fb)$ be
the compact forms of $\fm$ and $\fz(\fb)$. Then, 
\begin{align}
	\fu(\fb)=\sqrt{-1}\fb\oplus 
\fu_{\fm}. 
\end{align} 
Let $\fu^{\bot}(\fb)$ be the orthogonal space of $\fu(\fb)$ in $\fu$. 
Then, 
\begin{align}
&\fu^{\bot}(\fb)=\sqrt{-1}\fp^{\bot}(\fb)\oplus \fk^{\bot}(\fb),&	\fu=\sqrt{-1}\fb\oplus \fu_{\fm}\oplus \fu^{\bot}(\fb).
\end{align} 
\end{re}

Elements of $\fb$ act on $\fz^{\bot}(\fb)$ with semisimple real  
eigenvalues.  We fix an element $f_{\fb}\in \fb$, called positive, such that 
$\ad(f_{\fb})|_{\fz^{\bot}(\fb)}$ is invertible.  The choice of $f_{\fb}$ is irrelevant.  Let $\fn\subset 
\fz^{\bot}(\fb)$ (resp. $\ol{\fn}\subset \fz^{\bot}(\fb)$) be the 
direct sum of the eigenspaces of $\ad(f_{\fb})|_{\fz^{\bot}(\fb)}$ associated to the  positive (resp. negative) eigenvalues. Then,  
\begin{align}\label{eqZbnn}
&	\fz^{\bot}(\fb)=\fn\oplus \ol{\fn}, &\ol{\fn}=\theta\fn. 
\end{align} 
Clearly, $Z^{0}(\fb)$ acts on $\fn$, $\ol{\fn}$ and 
preserves the first decomposition in \eqref{eqZbnn}. 

\begin{prop}\label{propnng} The following statements hold.
		\begin{enumerate}[i)]
		\item\label{nng1}  The vector spaces $\fn, \ol{\fn}\subset \fg$ are Lie  subalgebras of  $\fg$, which have the same even dimension. 

		\item\label{nng2} The bilinear form $B$ vanishes on $\fn,\ol{\fn}$ and induces a $Z^0(\fb)$-isomorphism,
		\begin{align}\label{eq324}
	\ol{\fn}^{*}\simeq \fn. 
\end{align} 

		
\item\label{nng4} The actions of $M$ on 
$\fn,\ol{\fn},\fn^{*},\ol{\fn}^{*}$ are equivalent. For $0\l j\l \dim \fn$,  we have isomorphisms of 
	representations of $M$,
\begin{align}\label{eqVAMs1}
	\Lambda^{j}(\fn^{*})\simeq \Lambda^{\dim \fn-j}(\fn^{*}). 
\end{align}

\item\label{nng5}  The projections on $\fp,\fk$ map $\fn,\ol{\fn}$ into 
$\fp_{\fm}^{\bot},\fk_{\fm}^{\bot}$ isomorphically. 

\item\label{nng6} The actions of $K_M$ on 
$\fn,\ol{\fn},\fp_{\fm}^{\bot},\fk_{\fm}^{\bot}$ are equivalent.
	\end{enumerate} 
\end{prop} 
\begin{proof}
This is  \cite[Proposition 3.2, Corollary 3.3]{Shen20} 
(see also \cite[Proposition 3.10]{BS19}). 
\end{proof}

Let $R(K)$ be the representation ring of $K$. We can identify $R(K)$ 
with the subring of the $\Ad(K)$-invariant smooth functions on $K$ which is generated by the 
characters of finite dimensional complex representations of $K$. 

Similarly, we can define $R(T)$.  The Weyl group 
$W(T:K)=N_{K}(T)/T$ acts on $R(T)$. By \cite[Proposition VI.2.1]{BrockerDieck}, 
the restriction induces an isomorphism of rings
\begin{align}\label{eqRKRT}
	 R(K)\simeq  R(T)^{W(T:K)}. 
\end{align}

Since $K_{M}$ and $K$ have the same maximal torus $T$, 
the restriction induces  an  
injective morphism $R(K)\to R(K_{M})$ of rings.  Recall a key result 
established in \cite[Theorem  3.5, Corollary 3.6]{Shen20} (see also \cite[Theorem 
6.1, Corollary 6.12]{Shfried}). 

\begin{thm}\label{corkey}
	For $i, j \in \mathbf N$, the adjoint representations of $K_M$ on 
$\Lambda^{i}(\fp_{\fm,\bC}^{*})$ and $\Lambda^{j}(\fn^{*}_{\bC})$ have unique 
lifts in $R(K)$. 
\end{thm}

\subsection{Admissible metrics}\label{sadRepG}
Let  $\rho : G 
\to \GL(V)$ be a finite dimensional complex representation of $G$.  Set 
\begin{align}\label{rt=rt}
\rho^\theta=\rho\circ\theta.
\end{align}
Then $\rho^{\theta}:G\to \GL(V)$ is still a representation of $G$. 

\begin{prop}\label{propMtt}
	If $\delta(G)=0$,  we have an isomorphism of  
	representations of $G$, 
	\begin{align}\label{eqd0tth}
	\rho\simeq \rho^{\theta}.
	\end{align}
\end{prop}
\begin{proof}
	When $\delta(G)=0$, by \cite[Problem XII.10.14]{Knappsemi}, 
	there is $k_{0}\in K$, such that $\Ad(k_{0})=-1$ on $\fp$ and 
	$\Ad(k_{0})=1$ 
	on $\fk$. 	 By \eqref{eq:cartan2}, for $g\in G$, we have 
	$\theta(g)=k_{0}gk_{0}^{-1}$. Therefore, $\rho(k_{0}):V\to V$ 	
	is the required isomorphism  \eqref{eqd0tth}. 
%
%
	%
\end{proof}

\begin{defin}
	A Hermitian metric $\<,\>_{V}$ on $V$ is called admissible, if for  $u,v\in V$, $Y_{1}\in \fp$, $Y_{2}\in \fk$, we have 
	\begin{align}\label{eqadmissible}
&\left\<\rho(Y_{1})u,v\right\>_{V}=\left\<u,\rho(Y_{1})v\right\>_{V},&\left\<\rho(Y_{2})u,v\right\>_{V}=-\left\<u,\rho(Y_{2})v\right\>_{V}.
	\end{align}
\end{defin}
	
Assume that  $V$ has an admissible metric.
If $\ol{\rho}^{*}$ denotes the anti-dual representation of $\rho$, 
the admissible metric induces  an isomorphism of $G$-representations, 
	\begin{align}\label{eqrhothetarho}
		\rho^{\theta}\simeq \ol{\rho}^{*}. 
	\end{align}
If $W\subset V$ is a 
 $(\fg_{\bC},K)$-submodule of $V$, then the orthogonal space 
 $W^{\bot}\subset V$  is still a $(\fg_{\bC},K)$-submodule. 
 Moreover, by restrictions, $W$ and $W^{\bot}$ still have  
 admissible metrics.  In this way, we see that any finite dimensional 
 $G$-representation with an 
admissible metric is completely reducible, i.e., it can be 
decomposed as a direct sum of irreducible $G$-representations.

By \cite[Lemma 3.1]{MatsuchimaMurakami}, if $G$ is semisimple, any finite dimensional representation of $G$ 
has an admissible Hermitian metric. When $G$ is reductive and has a 
compact centre, we have a similar  result. 

\begin{prop}\label{proprepa}
	If $G$ has a  compact centre $Z_{G}$, then any finite 
	dimensional  complex  representation $\rho: G\to \GL(V)$ has an admissible Hermitian metric. 
\end{prop}
\begin{proof}
	Let $G_{\rm ss}\subset G$ be the connected Lie subgroup of $G$ 
	associated to the Lie algebra $\[\fg,\fg\]\subset \fg$. By 
	\cite[Corollary 7.11]{KnappLie}, $G_{\rm ss}$ is a closed subgroup 
	of $G$, which is semisimple and  $G=G_{\rm ss}\cdot Z^{0}_{G}.$ Let $U_{\rm ss}$ be the compact form of $G_{\rm ss}$. By Weyl's Theorem 
\cite[Theorem 4.69]{KnappLie}, the universal cover 
$\widetilde{U}_{\rm ss}$ of $U_{\rm ss}$ is still compact. Since 
$\widetilde{U}_{\rm ss}$ is 
simply connected, by Weyl's unitary trick, the group  
$\widetilde{U}_{\rm ss}$ 
acts on $V$ which is compatible with the action of $[\fg,\fg]$. 
Moreover, the $\widetilde{U}_{\rm ss}$-action commutes with the  
$Z_{G}^{0}$-action. Thus, the group 
$\widetilde{U}_{\rm ss}\times Z_{G}^{0}$ acts on $V$. Since 
$\widetilde{U}_{\rm ss}\times Z_{G}^{0}$ is 
compact, there is a $\widetilde{U}_{\rm ss}\times Z^{0}_{G}$-invariant Hermitian metric on $V$, which is the desired  admissible 
Hermitian metric.  
\end{proof}


\begin{re}
	If $G$ has a noncompact centre, Proposition \ref{proprepa} does 
	not hold. For example, when $G=\mathbf{R}$, the representation $x\in \mathbf{R}\to 
	\begin{pmatrix}
		1 & x \\
		0 & 1
	\end{pmatrix}\in \GL_{2}(\bC)
	$  is not completely reducible, so it does not have an  admissible 
	metric. 
\end{re}

\section{The zeta functions of Ruelle and Selberg}\label{Sselberg}
The purpose of this section is to introduce the 
zeta functions of Ruelle and of Selberg on  locally symmetric 
spaces. 

This section is organised as follows. In Section \ref{sec:sym}, we 
introduce the symmetric space $X=G/K$,  the 
$K$-principal bundle $p:G\to X$, and a Hermitian vector bundle   associated to a finite dimensional unitary  representation  of $K$. 



In Section \ref{sec:Gamma}, we introduce a discrete cocompact  subgroup 
$\Gamma\subset G$ of $G$, the corresponding 
locally symmetric space $Z=\Gamma\backslash X$, and a flat  vector bundle associated to a finite  dimensional representation  of $\Gamma$. 

Finally, in Sections \ref{SRuelle} and \ref{sSelbergzeta},  we introduce the 
zeta functions of Ruelle and of Selberg. We recall their  properties  established in \cite[Section 7]{Shfried} and \cite[Section 5]{Shen20}.

We use the notation in Section \ref{Sreductive}. 
\subsection{Symmetric spaces}\label{sec:sym}
Set $X=G/K$. Let $p:G\to X$ be the 
natural  projection. Then $p: G\to X$ is a 
$K$-principal bundle. 

The group $K$ acts isometrically on $\fp$. The tangent bundle of $X$ is given by
\begin{align}\label{eqTX}
TX=G\times_K \fp.
\end{align}
By \eqref{eqTX}, the scalar product $B|_{\fp}$ on $\fp$ induces a 
Riemannian metric $g^{TX}$ on $X$. Classically, 
$(X,g^{TX})$ has a parallel and nonpositive sectional curvature.


Let $\tau$ be an orthogonal (resp. unitary) representation of $K$ acting on a finite dimensional Euclidean (resp. Hermitian) space $E_\tau$. Set
\begin{align}
  \label{eq:Etau}
  \cE_\tau=G\times_K E_\tau.
\end{align}
Then $\cE_{\tau}$ is a Euclidean (resp. Hermitian) vector bundle on 
$X$. 

By \eqref{eq:Etau}, we have an identification
\begin{align}\label{eqCXEGEK}
	C^{\infty}(X,\cE_\tau)=C^\infty(G,E_\tau)^K.
\end{align} 
The group $G$ acts on the left on $C^{\infty}(X,\cE_\tau)$. 
Denote by $C^{\fg,X,\tau}$ the Casimir element of $G$ on 
$C^{\infty}(X,\cE_\tau)$. By \eqref{eq:Cg}, $C^{\fg,X,\tau}$ is a   
generalised Laplacian  on $X$ in the sense of \cite[Definition 
2.2]{BGV}, which is self adjoint with respect to the standard 
$L^{2}$-product (c.f. \eqref{eqL2prod}).

\subsection{Locally symmetric spaces}\label{sec:Gamma}
Let $\Gamma\subset G$  be a discrete  
cocompact subgroup of $G$. By \cite[Lemma 1]{Selberg60} (see also 
\cite[Proposition 3.9]{Ma_bourbaki}), $\Gamma$ 
contains only semisimple elements. 
Let $\Gamma_{e}\subset \Gamma$ be the subset of 
elliptic  elements, and let  $\Gamma_{+}=\Gamma-\Gamma_{e}$ be the subset of 
nonelliptic elements. 

The group $\Gamma$ acts isometrically on the left on $X$. 
Take 
\begin{align}
	Z=\Gamma\backslash X=\Gamma \backslash 
G/K. 
\end{align}
Then, $Z$ is a compact orbifold with Riemannian metric $g^{TZ}$. We denote by $\widehat{p}:\Gamma\backslash G\to Z$ and $\widehat{\pi}:X\to Z$ \index{P@$\widehat{p},\widehat{\pi}$}the natural projections, so that the diagram
\begin{align}
\begin{aligned}
\xymatrix{
G \ar[d]^p \ar[r] &\Gamma \backslash G\ar[d]^{\widehat{p}}\\
X \ar[r]^{\widehat{\pi}} &Z}
\end{aligned}
\end{align}
commutes.

From now on until Section \ref{S:rep}, we assume that 
$\Gamma$ is torsion free, i.e., $\Gamma_{e}=\{\rm id\}$. Then $Z$ is a 
connected closed orientable Riemannian locally symmetric manifold with nonpositive sectional curvature. Since $X$ is contractible,
  $\pi_1(Z)=\Gamma$ and $X$ is the universal cover of $Z$.   
In Section \ref{Snontorsionfree}, we will 
consider  the  case where $\Gamma$ is not torsion free.

The $\Gamma$-action on $X$ lifts to all the homogeneous Euclidean or 
Hermitian vector bundles $\cE_\tau$ on $X$ constructed in 
\eqref{eq:Etau}.  Then $\cE_\tau$ descends to a 
Euclidean or Hermitian  vector 
bundle on $Z$,
\begin{align}\label{eqFtau0}
	\cF_{\tau}=\Gamma\backslash \cE_{\tau}=\Gamma\backslash 
G\times_{K}E_{\tau}. 
\end{align}


If $r\in \mathbf N^{*}$, and if $\rho:\Gamma\to \GL_{r}(\bC)$ is a representation of $\Gamma$,  
let $F$ be the associated flat vector bundle on $Z$,
\begin{align}\label{eqFhol1}
	F=\Gamma\backslash (X\times \mathbf{C}^{r}). 
\end{align} 
  The group $\Gamma$ acts  on  $C^{\infty}(G,E_{\tau})^{K}$, as well  as  on $\bC^{r}$ by $\rho$.  We have the 
  identification 
  \begin{align}\label{eqC316}
  	C^{\infty}(Z,\cF_{\tau}\otimes 
  	F)=\(C^{\infty}(X,\cE_{\tau})\otimes \bC^{r}\)^{\Gamma}. 
  \end{align}
The Casimir operator $C^{\fg,X,\tau}\otimes {\rm id}$ preserves  the 
above invariant space. Its action  on $C^{\infty}(Z,\cF_{\tau}\otimes 
  	F)$ will be denoted by $C^{\fg,Z,\tau,\rho}$. If $\rho$ is 
	unitary, $C^{\fg,Z,\tau,\rho}$ is self-adjoint with respect to 
	the $L^{2}$-product on $C^{\infty}(Z,\cF_{\tau}\otimes 
  	F)$ induced by	the Hermitian metric on $\cE_{\tau}$, the standard 
	Hermitian metric on $\bC^{r}$, as well as  the Riemannian volume of
	$(Z,g^{TZ})$ 	(c.f. \eqref{eqL2prod}). When 		$\rho$ is the trivial representation, we denote it by 	$C^{\fg,Z,\tau}$.

\subsection{The Ruelle zeta function}\label{SRuelle}
Let us recall the definition of the Ruelle dynamical zeta function introduced by Fried \cite[Section 5]{Friedconj}. 

For $\gamma\in \Gamma$, set 
\begin{align}
	\Gamma(\gamma)=Z(\gamma)\cap \Gamma. 
\end{align}
By \cite[Lemma2]{Selberg60} (see also \cite[Proposition 
4.9]{Shfried}, \cite[Proposition 3.9]{Ma_bourbaki}), $\Gamma(\gamma)$ is cocompact in $Z(\gamma)$.

Let $[\Gamma_{+}]$ and $[\Gamma]$ be the sets of conjugacy classes in $\Gamma_{+}$ and $\Gamma$.  
If $\gamma\in \Gamma$, the associated conjugacy class in $\Gamma$ is 
denoted by 
$[\gamma]\in [\Gamma]$.\footnote{
The quantity  $\ell_{[\gamma]}$  depends only on  the conjugacy class of  $\gamma$ in $G$. So they are well defined on the conjugacy classes of $\Gamma$.  
}
If $[\gamma]\in [\Gamma]$, for all $\gamma'\in [\gamma]$,  the 
locally symmetric spaces 
\begin{align}\label{eqBgamma}
	\Gamma(\gamma')\backslash Z(\gamma')/K(\gamma')
\end{align}
are canonically diffeomorphic, and will be denoted by 
$B_{[\gamma]}$. 

By \cite[Proposition 5.15]{DuistermaatKolkVaradarajan}, the set of  
 nontrivial closed geodesics  on  $Z$ consists of a disjoint union 
\begin{align}\label{DKVB}
\coprod_{[\gamma]\in [\Gamma_{+}]}B_{[\gamma]}.
\end{align}
If $[\gamma]\in [\Gamma_{+}]$, all the  elements of 
$B_{[\gamma]}$ have the same length $\ell_{[\gamma]}>0$.
\index{B@$B_{[\gamma]}$} 

If $[\gamma]\in [\Gamma_{+}]$,  the geodesic flow induces a  locally free action of 
$\mathbb{S}^1$ on $B_{[\gamma]}$, so that $ 
B_{[\gamma]}/ \mathbb{S}^1$ is a closed orbifold. Let $\chi_{\rm 
		orb}\(B_{[\gamma]}/\mathbb{S}^{1}\)\in \mathbf{Q}$ be the 
		orbifold Euler characteristic number \cite{SatakeGaussB}. We refer the reader to \cite[Proposition 5.1]{Shfried} for an explicit formula for $\chi_{\rm 
		orb}\(B_{[\gamma]}/\mathbb{S}^{1}\)$. In particular,  
if $\delta(G)\g 2$, or if $\delta(G)=1$ and $\gamma$ can not be 
conjugate by an element of $G$ into the fundamental Cartan subgroup $H$, then  
\begin{align}\label{eqchi0}
\chi_{\rm 
		orb}\(B_{[\gamma]}/\mathbb{S}^{1}\)=0.
\end{align}

The $\mathbb{S}^1$-action on $B_{[\gamma]}$ is not necessarily effective. Let
\begin{align}
m_{[\gamma]}=\left|\ker\big(\bbS^1\to {\rm 
Diff}(B_{[\gamma]})\big)\right|\in \mathbf{N}^{*}
\end{align}
be the generic multiplicity.\index{M@$m_{[\gamma]}$}

Recall that $\rho:\Gamma\to \GL_{r}(\bC)$ is a representation of $\Gamma$.  By \cite[(4.4)]{Shen20}, there is $\sigma_{0}>0$ such that
\begin{align}
	\sum_{[\gamma]\in [\Gamma_{+}]}\frac{\left|\chi_{\rm 
		orb}\(B_{[\gamma]}/\mathbb{S}^{1}\)\right|}{m_{[\gamma]}} 
		\left|\Tr\[\rho(\gamma)\]\right| e^{-\sigma_{0} 
		\ell_{[\gamma]}}<\infty. 
\end{align}

\begin{defin}
	For $\Re(\sigma)\g \sigma_{0}$, the Ruelle dynamical zeta 
	function is defined  by 
	\begin{align}\label{defRrho}
		R_{\rho}(\sigma)=\exp\Big(\sum_{[\gamma]\in 
		[\Gamma_{+}]}\frac{\chi_{\rm 
		orb}\(B_{[\gamma]}/\mathbb{S}^{1}\)}{m_{[\gamma]}}\Tr\[\rho(\gamma)\]e^{-\sigma \ell_{[\gamma]}}\Big).
	\end{align}
\end{defin}

\begin{re}\label{reR=1}
	By \eqref{eqchi0}, if $\delta(G)\g 2$, the 
	Ruelle zeta function $R_\rho(\sigma)$ is the constant function 
	$1$. Moreover, if $\delta(G)=1$, then the sum on the right-hand 	side of \eqref{defRrho} can be reduced to 
	 a sum over 	$[\gamma]\in [\Gamma_{+}]$ such that $\gamma$ can be conjugate into $H$. 
\end{re}

\begin{thm}
	If $\dim Z$ is odd, the Ruelle zeta function $R_{\rho}$ has a meromorphic extension to 
$\sigma\in \bC$. 
\end{thm}
\begin{proof}
This is  \cite[Theorem 0.1 i)]{Shen20}. 
\end{proof} 

Let $\ol{R}_{\rho}$ be the meromorphic function 
defined for $\sigma\in \bC$ by
\begin{align}
	\ol{R}_{\rho}(\sigma)=\ol{R_{\rho}(\ol{\sigma})}.
\end{align}
By \cite[Proposition 4.4]{Shen20}, we have 
\begin{align}\label{eqRrhorho1}
&	R_{\rho^{*}}=R_{\rho},&R_{\ol{\rho}}=\ol{R}_{\rho}.
\end{align} 

If $\rho$ is a finite dimensional complex representation of $G$, the restriction 
$\rho_{|\Gamma}$ is a representation of $\Gamma$. We write 
$R_{\rho}=R_{\rho_{|\Gamma}}$ to ease the notation. 

\begin{prop}\label{propRoR}
	If $\rho$ is a finite dimensional complex representation of $G$ with an admissible metric,  then the following identity of  meromorphic functions on $\bC$ holds, 
\begin{align}\label{eqRoR}
	R_{\rho^{\theta}}=\ol{R}_{\rho}. 
\end{align}
\end{prop}
 \begin{proof}
	This is a consequence of \eqref{eqrhothetarho} and \eqref{eqRrhorho1}.
\end{proof}

\subsection{The Selberg zeta function}
\label{sSelbergzeta} 
In this subsection, we assume $\delta(G)=1$ and  we use the notation 
in Section \ref{sSpliting}.  Recall that $K_{M}$ is defined in 
\eqref{eqKM0}.  We have seen that  the morphism   $R(K)\to 
R({K_{M}})$ of rings is injective. 

\begin{ass}\label{ass}
	Assume that $\eta=\eta^{+}-\eta^{-}$ is a virtual    $M$-representation  
	on the   finite 	dimensional complex vector space $E_{\eta}^{+}-E_{\eta}^{-}$ 	such that 
\begin{enumerate}
	\item\label{ass1}  $\eta_{|K_{M}}=\eta^{+}_{|K_{M}}-\eta^{-}_{|K_{M}}\in R(K_{M})$ has a unique lift in $R(K)$.
	\item\label{ass2}  the Casimir $C^{\fu_{\fm}}$ of $\fu_{\fm}$ acts on 
	$\eta^{\pm}$ by the same 	scalar $C^{\fu_{\fm},\eta}\in \bR$. 
\end{enumerate}
\end{ass}

Following \cite[Definition 7.4]{Shfried} and \cite[Defintion 5.7]{Shen20}, let us define  the 
Selberg zeta function associated\footnote{\label{f4}A more general 
construction for the Selberg zeta function  is given in 
\cite{Shen20},  which is  associated  to $\eta$ and to an 
arbitrary representation of $\rho:\Gamma\to \GL_{r}(\bC)$.} to $\eta$.
Recall 
that $H=\exp(\fb)\times T$ is the fundamental Cartan subgroup of $G$. 
For $e^{a}k^{-1}\in H$,  we write $\gamma\sim e^{a}k^{-1}\in H$ if there is $g_{\gamma}\in G$ such 
that $\gamma=g_{\gamma}e^{a}k^{-1}g_{\gamma}^{-1}$. 
By \cite[(7-62)]{Shfried}, there is 
$\sigma_{1}>0$ such that
\begin{align}
	\sum_{\tiny\substack{[\gamma]\in 
		[\Gamma_{+}]\\ \gamma\sim e^{a}k^{-1}\in H}}\frac{\left|\chi_{\rm 
		orb}\(B_{[\gamma]}/\mathbb{S}^{1}\)\right|}{m_{[\gamma]}}  \frac{e^{-\sigma_{1} 
		\ell_{[\gamma]}}}{\left|\det\(1-\Ad(e^{a}k^{-1})\)|_{\fz^{\bot}(\fb)}\right|^{1/2}}<\infty. 
\end{align}

\begin{defin}
		For $\Re(\sigma)\g \sigma_{1}$, set 
	\begin{align}\label{eqdefsel}
		Z_{\eta}(\sigma)=\exp\Bigg(-\sum_{\tiny\substack{[\gamma]\in 
		[\Gamma_{+}]\\ \gamma\sim e^{a}k^{-1}\in H}}\frac{\chi_{\rm 
		orb}\(B_{[\gamma]}/\mathbb{S}^{1}\)}{m_{[\gamma]}}\frac{\Trs^{E_{\eta}}[k^{-1}]}{\left|\det\(1-\Ad(e^{a}k^{-1})\)|_{\fz^{\bot}(\fb)}\right|^{1/2}}e^{-\sigma \ell_{[\gamma]}}\Bigg).
		\end{align}
\end{defin}

Recall that by Corollary \ref{corkey},  
$\Lambda^{\scriptscriptstyle\bullet }(\fp_{\fm,\bC}^{*}) $ has a unique 
lift
in $R(K)$.  
\begin{defin}
Let
$\widehat{\eta}\in R(K)$ be the unique 
virtual representation of $K$ on $E_{\widehat{\eta}}=E^{+}_{\widehat\eta}-E^{-}_{\widehat\eta}$
such that the following identity in $R(K_M)$  holds, 
\begin{align}\label{eqetahat}
	E_{\widehat{\eta}{|K_{M}}}=\Lambda^{\scriptscriptstyle\bullet 
	}(\fp_{\fm,\bC}^{*})\widehat{\otimes} E_{\eta{|K_{M}}} \in R(K_{M}). 
\end{align}
\end{defin}

Let $C^{\fg,Z,\widehat{\eta}}$ be the self adjoint generalised 
Laplacian acting on $C^{\infty}(Z,\cF_{\widehat{\eta}})$ introduced  below  \eqref{eqC316}. For $\lambda\in \bC$, set 
\begin{align}
	m_{\eta}(\lambda)=\dim \ker 
	\(C^{\fg,Z,\widehat{\eta}^{+}}-\lambda\)-\dim \ker 
	\(C^{\fg,Z,\widehat{\eta}^{-}}-\lambda\). 
\end{align}

Let 
\begin{align}\label{eqdetCc}
\mathrm{det}_{\rm 
  gr}\(C^{\fg,Z,\widehat{\eta}}+\sigma\)=\frac{\mathrm{det}\big(C^{\fg,Z,\widehat{\eta}^{+}}+\sigma\big)}{\mathrm{det}\big(C^{\fg,Z,\widehat{\eta}^{-}}+\sigma\big)}
 \end{align} 
  be a 
  graded determinant of $C^{\fg,Z,\widehat{\eta}}+\sigma$. As in 
  \eqref{eqTsigma}, by \cite{Voros} (see also \cite[Theorem 1.5]{Shen20}), the function \eqref{eqdetCc} is 
  meromorphic on $\sigma\in \bC$. Moreover, its zeros and poles 
  belong to the  set $\left\{-\lambda: \lambda\in \Sp\( 
  C^{\fg,Z,\widehat{\eta}}\)\right\}$. 
  If $\lambda\in \Sp \(C^{\fg,Z,\widehat{\eta}}\)$, the order of 
  the zero  at $\sigma=-\lambda$ is $m_{\eta}(\lambda)$.

Following  \cite[(7-60)]{Shfried} and \cite[(5.18), (5.19)]{Shen20}, 
we set
\begin{align}\label{eqseta}
	\sigma_\eta=\frac{1}{8}\Tr^{\fu^\bot(\fb)}\[C^{\fu(\fb),\fu^{\bot}(\fb)}\]-C^{\fu_{\fm},\eta}. 
\end{align}
\begin{re}
When $Z_{G}$ is non compact,  by \cite[(4-52)]{Shfried}, we have 
\begin{align}\label{eqGbM}
	G=\exp(\fb)\times M. 
\end{align}
Therefore, $\fu^{\bot}(\fb)=0$. By  \eqref{eqseta},  we have
\begin{align}\label{eqsen1}
	\sigma_{\eta}=-C^{\fu_{\fm},\eta}.
\end{align}

Let $P_{\eta}(\sigma)$ be the odd polynomial  defined in 
\cite[(7-61)]{Shfried} and \cite[(5.20), Remark 5.9]{Shen20}.  

\begin{thm}\label{thm:detfor}
	The Selberg zeta function $Z_{\eta}(\sigma)$ has a meromorphic 
	extension to $\sigma\in \bC$ such that  the following identity of meromorphic functions on $\bC$ holds,
\begin{align}\label{eq:detfor}
  Z_{\eta}(\sigma)=\mathrm{det}_{\rm 
  gr}\(C^{\fg,Z,\widehat{\eta}}+\sigma_{\eta}+\sigma^2\)\exp\big(\vol(Z) P_{\eta}(\sigma)\big).
\end{align}
The zeros and poles of $Z_{\eta}(\sigma)$
belong to the set 
$\left\{\pm i\sqrt{\lambda+\sigma_{\eta}}:\lambda\in 
\Sp\(C^{\fg,Z,\widehat{\eta}}\)\right\}.
$
If $\lambda\in 
\Sp\(C^{\fg,Z,\widehat{\eta}}\)$ and 
$\lambda\neq -\sigma_{\eta}$, the order of zero  at $\sigma=\pm 
i\sqrt{\lambda+\sigma_{\eta}}$ is $m_{\eta}(\lambda)$. The order 
of zero at $\sigma=0$ is $2m_{\eta}(-\sigma_{\eta})$.
Also,
\begin{align}\label{eq:funceq}
  Z_{\eta}(\sigma)=Z_{\eta}(-\sigma)\exp\big(2\vol(Z) P_{\eta}(\sigma)\big).
\end{align}
\end{thm}
\begin{proof}
	This is \cite[Theorem 5.10]{Shen20} for the trivial twist (c.f. 
	Footnote \ref{f4}). 
\end{proof}

\end{re}

\section{The  Fried conjecture and  admissible 
metrics}\label{SFCAM}
In this section, we introduce a class of Hermitian flat vector 
bundles on locally symmetric spaces associated to 
representations of $G$ with an admissible metrics.  We state the main 
result (Theorem \ref{thm1}) of 
this article, which confirms the Fried conjecture for this 
class of Hermitian flat vector 
bundles. 

	This section is organised as follows. In Section 
	\ref{sFadmetric}, we introduce a Hermitian metric on a
	flat vector bundle whose holonomy representation  is the 
	restriction $\rho_{|\Gamma}$ of  a representation $\rho$ of $G$ 	with an admissible metric. 
	
	In Section \ref{sMR}, we state Theorem 	
	\ref{thm1}. 
	
	In Sections \ref{sC2}-\ref{sPC2}, we deduce Theorem \ref{Thm1},  
	Corollaries \ref{cori1} and \ref{cori2} from Theorem 	
	\ref{thm1}.

	Finally, in Section \ref{sproofth1cc}, we show Theorem 	
	\ref{thm1} when $\delta(G)=1$ and $Z_{G}$ is compact. 

\subsection{Hermitian metrics on flat vector bundles}\label{sFadmetric}
Let $\rho:G\to \GL(E)$ be a finite dimensional complex representation 
of $G$ with admissible Hermitian metric $\<,\>_{E}$.  

Let $F$ be the flat vector bundle associated to $\rho_{|\Gamma}$ 
defined in \eqref{eqFhol1}.
Let us construct a Hermitian metric $g^{F}$ on $F$ following  
\cite[Section 2.5]{Muller_3_torsion} and \cite[Section 8.1]{BMZ}. 
%
%
By the second identity of \eqref{eqadmissible},  the restriction 
$\rho_{|K}$ of $\rho$ to $K$ is  unitary.   Let 
\begin{align}
	\cE_{\rho_{|K}}=G\times_K E
\end{align}
 be the Hermitian vector bundle  on $X$ defined 
in \eqref{eq:Etau}. 
We have a  canonical  $G$-equivariant identification
\begin{align}\label{eq:flathomo}
 [g,v]\in  G\times_K E \to  (pg,gv)\in  X\times E. 
\end{align}
In this way, the $G$-invariant Hermitian metric  on 
$\cE_{\rho_{|K}}$  induces a $G$-invariant Hermitian metric 
$g^{\widehat{\pi}^{*}F}$
 on the trivial vector bundle $\widehat{\pi}^{*}F=X\times E$. 
It descends to a Hermitian metric $g^{F}$ on $F=\Gamma \backslash (X\times E)$.  

\begin{defin}
	We  will call such $(F,g^{F})$ an admissible Hermitian flat 
	vector 	bundle. The $g^{F}$ will be called an 
	admissible 	Hermitian metric on $F$. 
\end{defin}

By \eqref{eq:flathomo}, as in \eqref{eqCXEGEK} and \eqref{eqC316}, we have the identifications
\begin{align}\label{eqid55}
&	\Omega^{\scriptscriptstyle\bullet}\(X,\widehat{\pi}^{*}F\)=C^{\infty}\(G,\Lambda^{\scriptscriptstyle\bullet }(\fp^{*})\otimes_{\bR} E\)^{K},
	&\Omega^{\scriptscriptstyle\bullet 
	}(Z,F)=C^{\infty}\(\Gamma\backslash G,\Lambda^{\scriptscriptstyle\bullet }(\fp^{*})\otimes_{\bR} E\)^{K}.  
\end{align}

Let $\Box^{X}$ be the Hodge Laplacian on 
$X$ acting on $\Omega^{\scriptscriptstyle\bullet }(X,{\widehat{\pi}^{*}F})$ with 
respect to the  metrics $(g^{TX},g^{\widehat{\pi}^{*}F})$. 
Let 
$\Box^{Z}$ be the Hodge Laplacian on  $Z$ acting on $\Omega^{\scriptscriptstyle\bullet }(Z,F)$ with 
respect to the  metrics $(g^{TZ},g^{F})$.
Recall that $C^{\fu,\rho}\in \End(E)$ is the Casimir operator 
of $\fu$ acting on $E$ (see \eqref{eqCu=Cg}). The following proposition is classical (see \cite[Proposition 8.4]{BMZ}).

\begin{prop}\label{prop:D=C-C}
 	Under the identifications \eqref{eqid55}, we have  
  \begin{align}\label{eq:D=C-C}
&    \Box^X=C^{\fg,X,\Lambda^{\scriptscriptstyle\bullet 
	}(\fp^{*})\otimes_{\bR} E_{|K}}-C^{\fu,\rho},&    \Box^Z=C^{\fg,Z,\Lambda^{\scriptscriptstyle\bullet 
	}(\fp^{*})\otimes_{\bR}E_{|K}}-C^{\fu,\rho}.
  \end{align}
\end{prop}

Let $T(F)$ be the analytic torsion of $F$ associated to 
$(g^{TZ},g^{F})$. Let $N^{\Lambda^{\scriptscriptstyle\bullet 
}(T^*Z)}$ be the number operator, i.e. $N^{\Lambda^{\scriptscriptstyle\bullet 
}(T^*Z)}$ acts by multiplication by $k$ on $\Omega^k(Z,F)$.

\begin{thm}\label{thmMS0}
 Assume $\delta(G)\neq 1$. For any $t>0$, we have 
\begin{align}
	    \Trs\[\(N^{\Lambda^{\scriptscriptstyle\bullet }(T^*Z)}-\frac{m}{2}\)\exp\(-t\Box^Z\)\]=0.
\end{align}
In particular, 
\begin{align}\label{eq:tau=0}
    T(F)=1.
  \end{align}
\end{thm}
\begin{proof}
	This is \cite{BMZ1} and \cite[Theorem 8.6, Remark 8.7]{BMZ} (see also \cite[Theorem 5.5]{Ma_bourbaki}), which generalises  a 
	vanishing theorem originally due to \cite[Corollary 2.2]{MStorsion} (see also \cite[Theorem 
7.9.3]{B09}) where $F$ is assumed to be unitarily flat. 
\end{proof}


%

Recall that $\rho^{\theta}=\rho\circ \theta$. Clearly,  $\<,\>_{E}$ 
is still an  admissible metric for $\rho^{\theta}$.  Let 
$(F^{\theta},g^{F^{\theta}})$ be the  associated admissible Hermitian flat vector bundle on $Z$.  Since $\rho_{|K}=\rho_{|K}^{\theta}$, by \eqref{eqid55}, the Hodge  Laplacians of $(F,g^{F})$ and $(F^{\theta},g^{F^{\theta}})$ act on the same space. By \eqref{eq:D=C-C} and by  
$C^{\fu,\rho}=C^{\fu,\rho^{\theta}}$, these two Laplacians  coincide.  In particular,  
\begin{align}\label{eqHH}
	& H^{\scriptscriptstyle\bullet }\(Z,F\)\simeq  
	H^{\scriptscriptstyle\bullet }\(Z,F^{\theta}\),
&	T(F)=T\(F^{\theta}\). 
\end{align}

\subsection{The statement of the main result}\label{sMR}
The main result of this article  is the following.

\begin{thm}\label{thm1}
	Assume that $\dim Z$ is odd and that  $\rho:G\to \GL(E)$ is a 
	finite dimensional complex representation of $G$ with an admissible metric. 
	Let $(F,g^{F})$ be
	the associated  admissible Hermitian flat vector bundle. If $\rho \simeq \rho^\theta$, then 	
	there exist explicit  constants $C_{\rho}\in \bR^{*}$ 
	and $r_{\rho}\in \mathbf{Z}$ (see \eqref{eqTTb2} and Remark \ref{rerCr}) such that when $\sigma\to 0$, we have 
	\begin{align}\label{eq1rd}
		R_{\rho}(\sigma)=
		C_{\rho}T(F)^{2}\sigma^{r_{\rho}}+\cO(\sigma^{r_{\rho}+1}).
	\end{align}
	Moreover, if $H^{\scriptscriptstyle\bullet }(Z,F)=0$, then 
	\begin{align}\label{eq2rd}
&		C_{\rho}=1,&r_{\rho}=0,
	\end{align}
	so that 
	\begin{align}
		R_{\rho}(0)=T(F)^{2}. 
	\end{align}
\end{thm}
\begin{proof}
	Since $\dim Z$ is odd, $\delta(G)$ is odd. If $\delta(G)\g3$, by 
	Remark \ref{reR=1} and Theorem  \ref{thmMS0}, our theorem follows easily. If $\delta(G)=1$, 	we will 
	consider the case where $Z_G$ is non compact  in Section 
	\ref{sproofth1cc} and  the case where $Z_{G}$ is  compact  in Sections 	\ref{SPdn1} and \ref{S:rep}. 
\end{proof}

\begin{re}\label{re45}
Assume $\dim Z$ is even. Then $\delta(G)$ is even as well. If 
$\delta(G)\g 2$, then by Remark \ref{reR=1} and Theorem \ref{thmMS0},  we have $R_{\rho}(0)=T(F)^{2}=1$. If 
$\delta(G)=0$,  by the Theorem of Gauss-Bonnet-Chern and by 
\cite[Proposition 4.1]{Shfried}, we have 
\begin{align}
	\chi(Z,F)=\dim E \cdot  \chi(Z) \neq0.
\end{align} 
In particular, there are no acyclic  flat vector bundles.
\end{re}

\subsection{Proof of Corollary \ref{cori1}}\label{sC2}Let us restate 
Corollary \ref{cori1}. 
\begin{thm}\label{thm2}
	Assume that $\dim Z$ is odd and that  $\rho:G\to \GL(E)$ is a finite dimensional complex representation of $G$ with an admissible metric. 
	Let $(F,g^{F})$ be
	the associated  admissible Hermitian flat vector bundle.  If  
	$\rho$ is  irreducible and $\rho\not\simeq \rho^{\theta}$,  then 
\begin{align}\label{eqHvan}
H^{\scriptscriptstyle\bullet }(Z,F)=0,
\end{align}
and 
$R_{\rho}$ is regular at $\sigma=0$,	so that 
	\begin{align}\label{eqRabs}
		\left|R_{\rho}(0)\right|=T(F)^{2}. 
	\end{align}
\end{thm}
\begin{proof}
	The vanishing of the cohomology $H^{\scriptscriptstyle\bullet }(Z,F)$ is a consequence of \cite[Theorem 
	VII.6.7]{BW} (see \eqref{eqHZF=0} for a proof). 	
	
	Take 
	\begin{align}\label{eqrrr3}
	\rho'=\rho\oplus \rho^{\theta}.
	\end{align}
	By \eqref{eqRoR} and \eqref{eqrrr3}, we have 
	\begin{align}\label{eqRTP0}
		R_{\rho'}=R_{\rho}\ol{R}_{\rho}.
	\end{align}
	
Note that $\rho'$ has an admissible metric. Let $F'$ be the admissible Hermitian  flat vector bundle associated to $\rho'$. 
	By \eqref{eqHH} and \eqref{eqHvan}, we have
	\begin{align}\label{eqHFP0}
	&	H^{\scriptscriptstyle\bullet }\(Z,F'\)=H^{\scriptscriptstyle\bullet }(Z,F)\oplus 
		H^{\scriptscriptstyle\bullet }\(Z,F^{{\theta}}\)=0,
&		T(F')=T(F)T\(F^{\theta}\)=T(F)^{2}. 
	\end{align}
	
	Since $\rho'=\rho\oplus \rho^{\theta}$ is invariant by 
	$\theta$,	we can apply Theorem \ref{thm1} to $\rho'$. 
	Therefore, $R_{\rho'}$ is regular at $\sigma=0$ and 
	\begin{align}\label{eqRTpp}
		R_{\rho'}(0)=T(F')^{2}. 
	\end{align}
	By \eqref{eqRTP0}-\eqref{eqRTpp}, we get \eqref{eqRabs}.
\end{proof}

\subsection{Proof of Theorem \ref{Thm1}}
	Since any $G$-representation with an admissible metric can be decomposed as a direct sum of irreducible $G$-representations, 
	which still have admissible metrics, by Theorems \ref{thm1} and  
	\ref{thm2}, 	we get Theorem \ref{Thm1}.  \qed

\subsection{Proof of Corollary \ref{cori2}}\label{sPC2}
Let $\rho_{0}:\Gamma\to {\rm U}(r)$ be a unitary representation of 
$\Gamma$. Let $F_{0}$ be the associated flat vector bundle . Since 
$\rho_{0}$ is unitary, $F_{0}$ admits  a flat metric $g^{F_{0}}$. We will show that $(F_{0},g^{F_{0}})$ is indeed an admissible Hermitian flat 
vector bundle associated to a larger  reductive group. 

Let 
\begin{align}
&	\ul{G}=G\times {\rm U}(r),&\ul{K}=K\times {\rm U}(r). 
\end{align} 
Then, $\ul{G}$ is a connected real reductive group with maximal compact subgroup 
$\ul{K}$. We have an identification
\begin{align}
	\ul{G}/\ul{K}\simeq X. 
\end{align} 

Let $\ul{\Gamma}$ be the graph of $\rho_{0}$, i.e., 
\begin{align}
	\ul{\Gamma}=\left\{(\gamma,\rho_{0}(\gamma))\in \ul{G}:\gamma\in 
	\Gamma\right\}. 
\end{align} 
Then $\ul{\Gamma}$ is a discrete torsion free and cocompact subgroup of $\ul{G}$, so that 
\begin{align}
	\ul{\Gamma}\backslash \ul{G}/\ul{K}\simeq Z. 
\end{align} 

Let $\ul{\rho}_{0}:	\ul{G}\to {\rm U}(r)$ be the representation of $\ul{G}$ defined by the 
projection onto  the second component. The standard Hermitian metric 
on $\bC^{r}$ is admissible for the representation $\ul{\rho}_{0}$. Also, we have  an 
identification of  Hermitian flat vector bundles
\begin{align}
\ul{\Gamma}\backslash(	\ul{G}/\ul{K}\times \bC^{r})\simeq F_{0}. 
\end{align} 
In particular, $(F_{0},g^{F_{0}})$ is an admissible  Hermitian flat vector 
bundle associated to the representation $\ul{\rho}_{0}$ of $\ul{G}$. 

 More generally,  the admissible 
Hermitian flat vector bundle  $(F,g^{F})$ associated to the $G$-representation 
$\rho$ is also an admissible 
Hermitian flat vector bundle associated to the 
$\ul{G}$-representation $\ul{\rho}=\rho \boxtimes \mathbf 1$.   
Therefore, $F_{0}\otimes F$ with the  induced Hermitian metric is an 
admissible Hermitian flat vector bundle associated to 
$\ul{\rho}_{0}\otimes {\ul{\rho}}$. By these considerations, Corollary \ref{cori2} follows from Theorem \ref{Thm1} and Corollary \ref{cori1}. \qed

\subsection{Proof of Theorem \ref{thm1} when $\delta(G)=1$ and 
$Z_{G}$ is noncompact}\label{sproofth1cc}
Suppose now  $\delta(G)=1$ and  $Z_G$ is  non compact.
Let $\rho$ be a representation of $G$ with an admissible metric such that $\rho\simeq \rho^{\theta}$.  We can and we will assume that the Casimir $C^{\fu}$ acts on $\rho$ as a scalar $C^{\fu,\rho}\in \mathbf{R}$.  

For $\beta\in \fb^{*}_{\bC}$, denote by $\bC_{\beta}$ the one dimensional representation of 
$\exp(\fb)$ such that $a\in \fb$ acts as the scalar 
$\<\beta,a\>\in \bC$. Clearly, $\bC_{\beta}$ has an admissible metric 
if and only if $\beta\in \fb^{*}$. 

Recall that since $Z_G$ is not compact,  we have $G=\exp(\fb)\times 
M$ (see \eqref{eqGbM}).  Since a representation with an admissible metric is completely 
reducible, we can  write
\begin{align}\label{eqVAM1}
 \rho=\bigoplus_{\beta\in \fb^*}\bC_{\beta}\boxtimes  
  \eta_{\beta},
\end{align}
where $\eta_{\beta}$ is a representation of $M$. 
\begin{prop}\label{propetabii}
	The following statements hold.
		\begin{enumerate}[i)]
		\item  For $\beta\in \fb^{*}$, we have isomorphisms  of 
representations of $M$, 
\begin{align}\label{eqVAM2}
\eta_{-\beta}\simeq \eta_{\beta}. 
\end{align}
		\item For $\beta\in \fb^{*}$, the representation 
		$\eta_{\beta}$ of $M$ satisfies 	Assumption \ref{ass}, so 
		that 
		\begin{align} \label{eqCud1}
C^{\fu_{\fm},\eta_{\beta}}=C^{\fu,\rho}+|\beta|^{2}\in \bR,
\end{align} 
	\end{enumerate} 
\end{prop} 
\begin{proof}
Since 
$\rho^{\theta}\simeq \rho$, by \eqref{eqVAM1}, we have isomorphisms  of 
representations of $M$,  $
\eta_{\beta}^{\theta}\simeq \eta_{-\beta}.$ Since $\delta(M)=0$, by Proposition \ref{propMtt}, we have isomorphisms   of 
representations of $M$, $\eta_{\beta}^{\theta}\simeq 
\eta_{\beta}$. From these considerations, \eqref{eqVAM2} follows.  

By \eqref{eqGbM}, we have $K_{M}=K$, so Assumption \ref{ass} 
\eqref{ass1} is trivial. Since $C^{\fu,\rho}$ is a scalar, by 
	\eqref{eqGbM} and \eqref{eqVAM1}, the Casimir of $\fu_{\fm}$ acts on $\eta_{\beta}$ as  
a scalar given in \eqref{eqCud1}. In particular, 
$\eta_{\beta}$ satisfies also Assumption \ref{ass} \eqref{ass2}. 
\end{proof} 

Recall that in Section \ref{sSpliting}, we have fixed a positive 
element $f_{\fb}\in 
\fb$ in $\fb$. Set
	\begin{align}
		\fb_{+}^{*}=\{\alpha\in \fb^{*}: \<\alpha,f_{\fb}\>>0\}. 
	\end{align} 
By 
\eqref{eqVAM2}, we can rewrite 
\eqref{eqVAM1} as 
\begin{align}\label{eqVAM0d1}
	 \rho=\mathbf 1\boxtimes  \eta_{0}\oplus \bigoplus_{\beta\in \fb^*_{+}}\(\bC_{\beta}\oplus \bC_{-\beta}\)\boxtimes
  \eta_{\beta}.
\end{align}
Note that $\eta_{0}$ can be zero.

Let
$Z_{\eta_{\beta}}(\sigma)$ be Selberg zeta function associated to $\eta_{\beta}$.

\begin{prop}\label{propd1RZ}
	The following identity of meromorphic functions on $\bC$ holds, 
	\begin{align}\label{eqRZZ}
	R_{\rho}(\sigma)=\(Z_{\eta_{0}}(\sigma)\)^{-1}\prod_{\beta\in 
	\fb^{*}_{+}}\(Z_{\eta_{\beta}}\(\sigma+|\beta|\)Z_{\eta_{\beta}}\(\sigma-|\beta|\)\)^{-1}. 
\end{align}
\end{prop}
\begin{proof}
	By \eqref{eqVAM0d1}, for $e^{a}k^{-1}\in H$, we have 
\begin{align}\label{eqd1E2}
	\Tr\[\rho\(e^{a}k^{-1}\)\]
	=\Tr\[\eta_{0}\(k^{-1}\)\]+\sum_{\beta\in 
	\fb^{*}_{+}}\Tr\[\eta_{\beta}\(k^{-1}\)\]\(e^{|\beta||a|}+e^{-|\beta||a|}\).
\end{align}
By \eqref{eqGbM}, we have $\fz^{\bot}(\fb)=0$, so that the denominator  
$\left|\det(1-\Ad(\gamma))|_{\fz^{\bot}(\fb)}\right|^{1/2}$ in 
\eqref{eqdefsel} disappears. By  \eqref{defRrho}, \eqref{eqdefsel}, and \eqref{eqd1E2}, we get \eqref{eqRZZ}.
\end{proof}

%

  By \eqref{eqsen1}, \eqref{eq:detfor}, and \eqref{eqCud1}, we see 
  that  
\begin{align}\label{eqsumbetamad1}
	Z_{\eta_{\beta}}\(-\sqrt{\sigma^2+|\beta|^2}\)Z_{\eta_{\beta}}\(\sqrt{\sigma^2+|\beta|^2}\)=\[{\rm det}_{\rm gr}\(C^{\fg,Z,\widehat{\eta}_{\beta}}-C^{\fu,\rho}+\sigma^{2}\)\]^{2}
\end{align}
is a meromorphic function on $\bC$.  We have a generalisation of \cite[Theorem 5.6]{Shfried}. Recall that $T(\sigma)$ is defined in \eqref{eqTsigma}. 

\begin{thm}
  \label{prop:RT11}
  The following identity of meromorphic functions on $\bC$ holds, 
  \begin{align}\label{eq:RT11}
    R_\rho(\sigma)=T\(\sigma^2\)  \exp\(-\vol(Z)P_{\eta_{0}}(\sigma)\)
 \prod_{\beta\in \fb^{*}_{+}}  
\frac{Z_{\eta_{\beta}}\(-\sqrt{\sigma^2+|\beta|^2}\)Z_{\eta_{\beta}}\(\sqrt{\sigma^2+|\beta|^2}\)}{Z_{\eta_{\beta}}\big(\sigma+|\beta|\big)Z_{\eta_{\beta}}\big(\sigma-|\beta|\big)}.
  \end{align}
\end{thm}
\begin{proof}
By \eqref{eqGbM}, we have $\fp=\fp_{\fm}\oplus \fb$.  By \eqref{eqetahat}, we have an identity in $R(K)$, 
\begin{align}\label{eqrhobeta1}
\widehat{	
\eta}_{\beta}=\sum_{i=1}^{m}(-1)^{i-1}i\Lambda^{i}(\fp^{*}_{\bC})\otimes \eta_{\beta|K}. 
\end{align} 
By \eqref{eq:D=C-C}, \eqref{eqVAM0d1}, and \eqref{eqrhobeta1},  
we have an identity of meromorphic functions,
\begin{align}\label{eq:detfortor11}
  T(\sigma)
  ={\rm 
  det}_{\rm 
  gr}\(C^{\fg,Z,\widehat{\eta}_{0}}-C^{\fu,\rho}+\sigma\)^{-1}\prod_{\beta\in \fb^{*}_{+}}{\rm 
  det}_{\rm 
  gr}\(C^{\fg,Z,\widehat{\eta}_{\beta}}-C^{\fu,\rho}+\sigma\)^{-2}.
\end{align}

%

By \eqref{eqsen1}, \eqref{eq:detfor},  \eqref{eqCud1},  \eqref{eqsumbetamad1},  \eqref{eq:detfortor11},  we have
\begin{multline}\label{eq:thtoZ11}
  T\(\sigma^2\)=
Z_{\eta_{0}}(\sigma)^{-1}\exp\big(\vol(Z)P_{\eta_{0}}(\sigma)\big)\\
\times \prod_{\beta\in\fb_{+}^{*}}  
\(Z_{\eta_{\beta}}\(-\sqrt{\sigma^2+|\beta|^2}\)Z_{\eta_{\beta}}\(\sqrt{\sigma^2+|\beta|^2}\)\)^{-1}.
\end{multline}
By  \eqref{eqRZZ} and \eqref{eq:thtoZ11}, we get \eqref{eq:RT11}. 
\end{proof}

Set
\begin{align}\label{eqrrhobeta}
	r_{\eta_{\beta}}=\dim \ker 
	\(C^{\fg,Z,\widehat{\eta}_{\beta}^{+}}-C^{\fu,\rho}\)-\dim \ker 
	\(C^{\fg,Z,\widehat{\eta}_{\beta}^{-}}-C^{\fu,\rho}\). 
\end{align}
Following \cite[(7-75)]{Shfried}, put 
\index{C@$C_\rho$}\index{R@$r_\rho$}
\begin{align}\label{eqTTb2}
 & 
 C_\rho=\prod_{\beta\in 
 \fb_{+}^{*}}\big(-4|\beta|^2\big)^{-r_{\eta_\beta}},&r_\rho=-2\sum_{\beta\in \{0\}\cup \fb^{*}_{+}}r_{\eta_{\beta}}.
\end{align}
Proceeding as  
\cite[(7-76)-(7-78)]{Shfried},  using  Theorem \ref{prop:RT11} instead of \cite[Theorem 
7.8]{Shfried}, we get 	\eqref{eq1rd}.

Let $F_{\beta}$ be the admissible Hermitian flat subbundle of $F$ associated to the 
	$G$-representation  $\bC_{\beta}\boxtimes \eta_{\beta}$. By 
	\eqref{eqEuler}, \eqref{eq:D=C-C}, \eqref{eqrhobeta1}, and \eqref{eqrrhobeta}, we have   
\begin{align}\label{eqrbetad1}
r_{\eta_\beta}=-	\chi'(Z,F_{\beta}).
\end{align} 
If $H^{\scriptscriptstyle\bullet }(Z,F)=0$, then for all $\beta\in 
\fb^{*}$,  $H^{\scriptscriptstyle\bullet }(Z,F_{\beta})=0$, so  $r_{\eta_\beta}=0$. By \eqref{eqTTb2}, we get \eqref{eq2rd}.

The proof of Theorem \ref{thm1} in the case when $\delta(G)=1$ and $Z_{G}$ is non compact  is completed. 
\qed

\section{The proof of (\ref{eq1rd})  when 
$\delta(G)=1$ and $Z_{G}$ is compact}
\label{SPdn1}
The purpose of this section is to show (\ref{eq1rd})  when 
$\delta(G)=1$ and $Z_{G}$ is compact by generalising the arguments 
given in Section \ref{sproofth1cc}. One of the difficulties is to 
construct virtual representations $\eta_{\beta}$  of $M$
satisfying Assumption \ref{ass} such that  an analogue of 
\eqref{eqRZZ} holds. In the case of hyperbolic manifolds, such representations are 
constructed using $\fn$-cohomology \cite{Muller_3_torsion}. Here, we construct 
$\eta_{\beta}$ via Dirac cohomology. These two methods are 
equivalent. We adopt the latter since it is closer to certain 
constructions given in \cite[Section 6]{Shfried}. 

This section is organised as follows. In Section \ref{srank1}, we 
recall some facts on the structure of real reductive groups with  $\delta(G)=1$ established in \cite[Section 6]{Shfried}. 

In Section \ref{srho2rho2}, we decompose  $\rho$ according to the action 
of $\fb$.  We show certain representations of $K_{M}$ obtained in this way can be lifted in $R(K)$. 

In Section \ref{srepetab}, we introduce  virtual representations $\eta_{\beta}$ of $M$ satisfying Assumption \ref{ass}. 

Finally, in Section \ref{sZetaetabeta}, we establish analogues of 
Proposition \ref{propd1RZ} and Theorem \ref{prop:RT11}, and we show (\ref{eq1rd}). 

In this section, we assume that  $\delta(G)=1$ and $Z_G$ is  compact. Suppose also  that  the $G$-representation $(\rho,E)$ has an admissible metric and is such that  $\rho\simeq 
\rho^\theta$.  As in Section \ref{sproofth1cc},  we can and we will 
assume that Casimir operator $C^{\fu}$ acts as a scalar  
$C^{\fu,\rho}\in \mathbf{\bR}$.

\subsection{The structure of the reductive group $G$ with 
$\delta(G)=1$}\label{srank1}Since $G$ has a compact centre and $\dim 
\fb=1$, we have $\fb\not\subset \fz_{\fg}$, so $\fn\neq0$.  By Proposition \ref{propnng} \ref{nng1}), $\dim \fn$ is a positive 
even number. Set
\begin{align}
  \ell=\frac{1}{2}\dim \fn\in \mathbf N^{*}.\index{L@$\ell$}
\end{align}

\begin{prop}\label{prop:nnam1}
	Elements of $\fb$ act on $\fn$ and $\ol{\fn}$ as a scalar, i.e., 	there is $\alpha_{0}\in \fb^{*}$ such that for $a\in 
\fb$, $f\in \fn$, $\ol{f}\in \ol{\fn}$, we have
\begin{align}\label{eq:ab1}
  &\[a,f\]=\<\alpha_{0},a\>f, &\[a,\ol{f}\]=-\<\alpha_{0},a\>\ol{f}.
\end{align}
In particular,  
\begin{align}\label{eq:nnam}
&\[\fn,\ol{\fn}\]\subset \fz(\fb),&[\fn,\fn]=\[\ol{\fn},\ol{\fn}\]=0,
\end{align}
and 
\begin{align}\label{eq:fzbot1}
&\[\fz(\fb),\fz(\fb)\]\subset 
\fz(\fb),&\[\fz(\fb),\fz^\bot(\fb)\]\subset 
\fz^{\bot}(\fb),&&\[\fz^{\bot}(\fb),\fz^\bot(\fb)\]\subset \fz(\fb),\\
&\[\fu(\fb),\fu(\fb)\]\subset \fu(\fb),& 
\[\fu(\fb),\fu^\bot(\fb)\]\subset \fu^{\bot}(\fb),&&  
\[\fu^{\bot}(\fb),\fu^\bot(\fb)\]\subset \fu(\fb).\notag
\end{align}
\end{prop}
\begin{proof}
	This is  \cite[Propositions 6.2, 6.3, and (6-29)]{Shfried}. 
\end{proof}

 Let $a_{0}\in \fb$ be such that 
 \begin{align}
 \<	\alpha_{0},a_{0}\>=1. 
 \end{align}

By \eqref{eq:fzbot1}, $(\fu,\fu(\fb))$ is a compact symmetric pair. 
Recall that $G$ has a compact centre.  Let $U$ be the compact form of $G$ \cite[Proposition 5.3]{Knappsemi}. Let $U(\fb)$ be the centraliser of $\fb$ in $U$. Then, 
$U(\fb)$ is a connected Lie group \cite[Corollary 4.51]{KnappLie} with Lie algebra $\fu(\fb)$. 

In \cite[(6-31)]{Shfried}, we have shown that 
$J=\ad\(\sqrt{-1}a_{0}\)\in \End(\fu^{\bot}(\fb))$
defines a  $U(\fb)$-invariant complex structure on 
$\fu^{\bot}(\fb)$. Moreover, the associated holomorphic and 
anti-holomorphic subspaces  are $\fn_{\bC}$ 
and $\ol{\fn}_{\bC}$, so that we have a $U(\fb)$-equivariant splitting 
\begin{align}\label{equbnn1}
	\fu^{\bot}(\fb)\otimes_{\bR}{\bC}=\fn_{\bC} \oplus \ol{\fn}_{\bC}.
\end{align} 
Let $S^{\fu^\bot(\fb)}$ \index{S@$S^{\fu^\bot(\fb)}$} be the spinor 
of $(\fu^\bot(\fb), -B|_{\fu^\bot(\fb)})$. Classically  
(\cite{Hitchin74}, see also \cite[(6-33) and (6-34)]{Shfried}), we have an isomorphism of $U(\fb)$-representations, 
\begin{align}\label{eq:Rnspinc}
	S^{\fu^\bot(\fb)}_{\pm}\simeq  \Lambda^{\rm even/odd 
	}\(\ol{\fn}_\bC^*\)\otimes 
  \det\(\fn_\bC\)^{-1/2}. 
\end{align} 

In the sequel, there are representations which do not always lift  to 
$U(\fb)$. Therefore, it is more convenient to consider $S^{\fu^\bot(\fb)}$ as a $(\fb_{\bC}\oplus \fm_{\bC}, 
K_{M})$-module. 

For $0\l j\l 2\ell$, let  $\eta_{j}$ \index{E@$\eta_{j}$} be the 
$(\fm_{\bC},K_{M})$-module $\Lambda^{j}(\fn^{*}_{\bC})$. By 
\eqref{eqVAMs1}, we have an isomorphism of $(\fm_{\bC},K_{M})$-modules, 
\begin{align}\label{eqVAMs11}
	\eta_{\ell-j}\simeq \eta_{\ell+j}. 
\end{align}


\begin{prop}\label{propSpinor}
	We have an isomorphism of $(\fb_{\bC}\oplus \fm_{\bC},K_{M})$-modules,
\begin{align}\label{eq:Rnspinc1}
  S^{\fu^\bot(\fb)}\simeq \bigoplus_{j=-\ell }^{\ell}\bC_{j\alpha_{0}} 
	\boxtimes \eta_{\ell-j}.
\end{align}
%
For $e^{a}k^{-1}\in H$, we have
\begin{align}\label{eqSnbot}
	\Trs^{S^{\fu^\bot(\fb)}}\[e^{a}k^{-1}\]=\left|\det\(1-\Ad(e^{a}k^{-1})\)|_{\fz^{\bot}(\fb)}\right|^{1/2}.
\end{align}
\end{prop}
\begin{proof}
	By  \eqref{eq:ab1},  
\eqref{eq:Rnspinc}, and \eqref{eqVAMs11},  we get  
\eqref{eq:Rnspinc1}. By \eqref{eq:Rnspinc1} and \cite[Proposition 
6.5]{Shfried}, we get \eqref{eqSnbot}. 
\end{proof}

Let $B^{*}$   be the bilinear form on $\fg^{*}$ induced by $B$. 
\begin{prop}
	We have
\begin{align}\label{eqlla0}
\frac{1}{8}\Tr\[C^{\fu(\fb),\fu^{\bot}(\fb)}\]=-\ell^{2}B^{*}\(\alpha_{0},\alpha_{0}\). 
\end{align} 
\end{prop} 
\begin{proof}
This is \cite[Proposition 6.13]{Shfried} with $j=0$.  A direct proof 
of this result can be obtained by applying Kostant's stranger  formula (see also \cite[(2.6.11), (7.5.4)]{B09}), which is left to reader. 
\end{proof}

\subsection{A splitting of $\rho$ according to the $\fb$-action}\label{srho2rho2}
Recall that $Z^{0}(\fb)=\exp(\fb)\times M$. 
Since $\rho$ has an admissible metric and since  $\rho \simeq \rho^{\theta}$, as in \eqref{eqVAM1} and  \eqref{eqVAM0d1},  we can 
write
\begin{align}\label{eqVAM0}
	 \rho_{{|\exp(\fb)\times M}}=\bigoplus_{\beta\in \fb^*}\bC_{\beta}\boxtimes  
  \rho_{\beta}=\mathbf{1}\boxtimes  \rho_{0}\oplus \bigoplus_{\beta\in \fb^*_{+}}\(\bC_{\beta}\oplus \bC_{-\beta}\)\boxtimes
  \rho_{\beta}.
\end{align}
where $\rho_{\beta}$ are representations of $M$  such that $\rho_{-\beta}\simeq \rho_{\beta}$. 


%
Recall that the restriction induces an injective  morphism 
$R(K)\to R(K_M)$ of rings. 

\begin{thm}\label{thm:key}
	For all $\beta\in \fb^{*}$, 
  the restriction  $\rho_{\beta|{K_{M}}}$ has a unique lift in $R(K)$.
\end{thm}
\begin{proof}
	By \eqref{eqRKRT}, it is enough to show that  the character 
	of $\rho_{\beta|{T}}$ is invariant under  the Weyl group $W(T:K)$.  

For $e^{a}k^{-1}\in H$, by in \eqref{eqVAM0}, as \eqref{eqd1E2}, we have  
\begin{align}\label{eqVAM}
		\Tr\[\rho\(e^{a}k^{-1}\)\]=\Tr\[\rho_{0}\(k^{-1}\)\]+\sum_{\beta\in 
		\fb^{*}_{+}}\(e^{|\beta||a|}+e^{-|\beta||a|}\)\Tr\[\rho_{\beta}\(k^{-1}\)\]. 
	\end{align}
Let $w\in N_K(T)$. 
Since 
$\rho$ is a $G$ representation,  for $e^{a}k^{-1}\in H$, we have 
\begin{align}\label{eqplif11}
\Tr\[\rho\(e^{a}k^{-1}\)\]=	
\Tr\[\rho\(we^{a}k^{-1}w^{-1}\)\]=\Tr\[\rho\(e^{\Ad(w)a}wk^{-1}w^{-1}\)\].
\end{align}
By \eqref{eqh=bt}, $\Ad(w)$ preserves $\ft$ and $\fb$ (see 
\cite[Proposition 3.4]{Shen20}). Since  $K$ preserves 
$B|_{\fp}$, and since $\dim \fb=1$, we see that  $\Ad(w)a=a$ or $-a$. By 
\eqref{eqVAM}, we get
\begin{multline}\label{eqplif1}
	\Tr\[\rho\(e^{\Ad(w)a}wk^{-1}w^{-1}\)\]=\Tr\[\rho_{0}\(wk^{-1}w^{-1}\)\]\\
	+ \sum_{\beta\in 
		 \fb^{*}_{+}}\(e^{|\beta||a|}+e^{-|\beta||a|}\)\Tr\[\rho_{\beta}\(wk^{-1}w^{-1}\)\].
\end{multline}	
By \eqref{eqVAM}-\eqref{eqplif1}, since $\dim \fb=1$, for  $\beta\in 
\fb^{*}$, we have
\begin{align}
	\Tr\[\rho_{\beta}\(wk^{-1}w^{-1}\)\]=\Tr\[\rho_{\beta}\(k^{-1}\)\],
\end{align} 
i.e., the character of $\rho_{\beta|T}$ is invariant under the Weyl 
group $W(T:K)$.  
\end{proof}

\subsection{The representation $\eta_{\beta}$}\label{srepetab}
In \cite{Shfried}, we have shown  that $\eta_{j}$ satisfies Assumption 
\ref{ass} and we consider the  Selberg zeta function associated to $\eta_{j}$.  A naive way to 
generalise the arguments  in Section \ref{sproofth1cc} is to consider  the Selberg zeta function associated to   $\eta_{j}\otimes \rho_{\beta}$. 
However,  $\eta_{j}\otimes \rho_{\beta}$ 
satisfies  Assumption \ref{ass} (\ref{ass1}), while Assumption 
\ref{ass} (\ref{ass2}) fails in general. We need consider all the $j$ 
together to produce a virtual representation $\eta_{\beta}$. 
Recall that  $\rho:G\to \GL(E)$ is a representation of $G$, and  
that $S^{\fu^{\bot}(\fb)}$ is the spinor of $\(\fu^{\bot}(\fb), -B|_{\fu^{\bot}(\fb)}\)$. 
%
%
%
%
Let $D^{S^{\fu^{\bot}(\fb)}\otimes 		E}$ be the Dirac operator defined in Section \ref{sDirac}. Let $D_{\pm}^{S^{\fu^{\bot}(\fb)}\otimes 
		E}$ be the restriction of $D^{S^{\fu^{\bot}(\fb)}\otimes 		
		E}$ to $S_{\pm}^{\fu^{\bot}(\fb)}\otimes 		E$.

By \eqref{eqD2cp}, we have 
\begin{align}\label{eqD2cpu}
		-\(D^{S^{\fu^{\bot}(\fb)}\otimes 		
		E}\)^{2}=C^{\fu,\rho}+\frac{1}{8}\Tr\[C^{\fu(\fb),\fu^{\bot}(\fb)}\]-C^{\fu(\fb),S^{\fu^{\bot}(\fb)}\otimes 
	E}.
	\end{align}
Since $\fu$ acts unitarily on $(E,\<,\>_{E})$	, we  have
\begin{align}\label{eqD2cpu23}
	\ker D^{S^{\fu^{\bot}(\fb)}\otimes 		
		E}=\ker \(D^{S^{\fu^{\bot}(\fb)}\otimes 		
		E}\)^{2}. 
\end{align} 
If the $\fu(\fb)$-action on $E$ lifts to $U(\fb)$, then $\ker 
D_\pm^{S^{\fu^{\bot}(\fb)}\otimes 		E}$ are  
representations\footnote{They are called 
Dirac cohomology of $E$ (see \cite{HuangDiracCoh}).} of $U(\fb)$. In general, $\ker 
D_\pm^{S^{\fu^{\bot}(\fb)}\otimes 		E}$ are  
	$(\fb_{\bC}\oplus \fm_{\bC},K_{M})$-modules. Note that by 
	\eqref{eq:Rnspinc1} and \eqref{eqVAM0}, $\fb$ acts semisimplely on 
	$\ker 
D_\pm^{S^{\fu^{\bot}(\fb)}\otimes 		E}$. 

\begin{defin}
	Define the $(\fm_{\bC},K_{M})$-modules  $\eta^{+}_{\beta}$ and 
	$\eta^{-}_{\beta}$, so that  
	\begin{align}\label{eqdetapm}
&		\ker D_{+}^{S^{\fu^{\bot}(\fb)}\otimes 
		E}=\bigoplus_{\beta\in\fb^{*}}\mathbf{\bC}_{\beta}\boxtimes 
		\eta_{\beta}^{+},& \ker D_{-}^{S^{\fu^{\bot}(\fb)}\otimes 
		E}=\bigoplus_{\beta\in\fb^{*}}\mathbf{\bC}_{\beta}\boxtimes 
		\eta_{\beta}^{-}. 
	\end{align}
	Set
	\begin{align}
		\eta_{\beta}=\eta_{\beta}^{+}-\eta_{\beta}^{-}. 
	\end{align}
\end{defin}


Recall that $C^{\fu,\rho}\in \bR$ is a scalar.

\begin{prop}\label{propkey1}
	The following statements hold.
	\begin{enumerate}[i)]
		\item\label{keyp1}  We have  isomorphisms of $(\fm_{\bC},K_{M})$-modules,
		\begin{align}\label{eqAS771}
	\eta_{\beta}^{\pm}\simeq \eta_{-\beta}^{\pm}. 
\end{align} 
	
		\item\label{keyp2}  The virtual $(\fm_{\bC},K_{M})$-modules 
		$\eta_{\beta}$ satisfy Assumption \ref{ass}, so that 
				  the Casimir of $\fu_{\fm}$ acts on $\eta_\beta^{\pm}$ by the same scalar  
\begin{align}\label{eqAS2}
		C^{\fu_{\fm},\eta_{\beta}}=|\beta|^{2}+C^{\fu,\rho}+\frac{1}{8}\Tr^{\fu^{\bot}(\fb)}\[C^{\fu(\fb),\fu^{\bot}(\fb)}\]\in \bR. 
	\end{align}
%
%
%
	\end{enumerate}
	\end{prop}
\begin{proof}
	 By \eqref{eqD2cpu}-\eqref{eqdetapm},
	on $\mathbf{C}_{\beta}\boxtimes \eta_{\beta}^{\pm}$, we have
	\begin{align}\label{eqAS35}
		C^{\fu,\rho}+\frac{1}{8}\Tr\[C^{\fu(\fb),\fu^{\bot}(\fb)}\]-C^{\fu(\fb),S^{\fu^{\bot}(\fb)}\otimes 
	E}=0.
	\end{align}
On $\mathbf{C}_{\beta}\boxtimes \eta_{\beta}^{\pm}$, we have 
	\begin{align}\label{eqAS3}
		C^{\fu(\fb),S^{\fu^{\bot}(\fb)}\otimes 
		E}=-|\beta|^{2}+C^{\fu_{\fm},\eta_{\beta}^{\pm}}. 
	\end{align} 
	By \eqref{eqAS35} and \eqref{eqAS3},  we 
	see that $C^{\fu_{\fm},\eta_{\beta}^{\pm}}$ coincide and are given by 	\eqref{eqAS2}.	
	
By \eqref{eqVAMs11}, \eqref{eq:Rnspinc1}, and \eqref{eqVAM0}, we see 
that $S^{\fu^{\bot}(\fb)}_{\pm}\otimes E$ admits a decomposition  
like \eqref{eqVAM0}. By \eqref{eqD2cpu} and \eqref{eqD2cpu23}, $\ker 
D_{\pm}^{S^{\fu^{\bot}(\fb)\otimes E}}$ consists of the components on 
which \eqref{eqAS2} holds. Form these two considerations,   
\ref{keyp1}) follows. 
	
It remains to show each $\eta_{\beta|{K_{M}}}$ lifts to $R(K)$.  By \eqref{eqdetapm},	 for $e^{a}k^{-1}\in H$, we have
	\begin{align}\label{eqAS5}
		\sum_{\beta\in 
		\fb^{*}}e^{\<a,\beta\>}\Tr_{\rm 
		s}\[\eta_{\beta}(k^{-1})\]=\Trs^{\ker \big(D^{S^{\fu^{\bot}(\fb)}\otimes 
	E}\big)}\[e^{a}k^{-1}\]. 
	\end{align}
	
	We claim that for $e^{a}k^{-1}\in H$, we have 
\begin{align}\label{eqAS4}
	\Trs^{\ker \big(D^{S^{\fu^{\bot}(\fb)}\otimes 
	E}\big)}\[e^{a}k^{-1}\]= \Tr_{\rm s}^{S^{\fu^\bot(\fb)}\otimes 
	E}\[e^{a}k^{-1}\]. 
\end{align}
Indeed, 	since $D^{S^{\fu^{\bot}(\fb)}\otimes 
	E}$ commutes with $e^{a}k^{-1}$, using the fact that the 
	super trace varnishes  on  the super commutator, we see that 
	\begin{align}
		\Tr_{\rm s}^{S^{\fu^\bot(\fb)}\otimes 
	E}\[e^{a}k^{-1}\exp\(-t\(D^{S^{\fu^{\bot}(\fb)}\otimes 
	E}\)^{2}\)\]
	\end{align}
does not depend on $t\in \bR$, from which we get \eqref{eqAS4}. 

From \eqref{eqAS5} and \eqref{eqAS4}, for $e^{a}k^{-1}\in H$, we have
	\begin{align}\label{eqAS33}
		\sum_{\beta\in 
		\fb^{*}}e^{\<\beta,a\>}\Tr_{\rm 
		s}\[\eta_{\beta}(k^{-1})\]=\Trs^{S^{\fu^\bot(\fb)}}\[e^{a}k^{-1}\]\Tr\[\rho\(e^{a}k^{-1}\)\]. 
	\end{align}
By Theorems  \ref{corkey}, \ref{thm:key}, \eqref{eq:Rnspinc1}, \eqref{eqVAM0}, the right-hand side of \eqref{eqAS33} is a sum 
of products  of $e^{\<\beta,a\>}$ with  elements in $R(K)\simeq R(T)^{W(T:K)}$. Thus, 
$\eta_{\beta}$ has a lift in $R(K)$. 
\end{proof}

The following two propositions are  analogues  of \cite[Propositions 
6.5, 6.10]{Shfried}.

\begin{prop}
	We have an isomorphism of virtual  
	$(\fb_{\bC}\oplus\fm_{\bC},K_{M})$-modules,
\begin{align}\label{eqSrhoeta1}
\(S^{\fu^{\bot  }(\fb)}_{+}-S^{\fu^{\bot  }(\fb)}_{-}\)\otimes  
\rho_{|\exp(\fb)\times M}=\bigoplus_{\beta\in \fb^{*}} \bC_{\beta}\boxtimes \eta_{\beta}=\mathbf 1\boxtimes \eta_{0}\oplus 
\bigoplus_{\beta\in \fb_{+}^{*}} \(\bC_{\beta}\oplus 
\bC_{-\beta}\)\boxtimes \eta_{\beta}.
\end{align} 
	For $e^{a}k^{-1}\in H$, 	we have
\begin{multline}\label{eqAS333}
	\Tr_{\rm 
		s}\[\eta_{0}\(k^{-1}\)\]+	\sum_{\beta\in 
		\fb^{*}_{+}}\(e^{|a||\beta|}+e^{-|a||\beta|}\)\Tr_{\rm 
		s}\[\eta_{\beta}\(k^{-1}\)\]\\
		=\left|\det\(1-\Ad(e^{a}k^{-1})\)|_{\fz^{\bot}(\fb)}\right|^{1/2}\Tr\[\rho\(e^{a}k^{-1}\)\] 
	\end{multline}
\end{prop}
\begin{proof}
	This is a consequence of \eqref{eqSnbot}, \eqref{eqAS771}, and \eqref{eqAS33}. 
\end{proof}

\begin{prop}
The following identity in $R(K)$ holds, 
\begin{align}\label{eqsumbeta}
		\bigoplus_{\beta\in 
		\fb^{*}}\widehat{\eta}_{\beta}=\bigoplus_{i=1}^{m}(-1)^{i-1}i 
		\Lambda^{i}\(\fp^{*}_{\bC}\)\otimes \rho_{|K}. 
	\end{align}
\end{prop}
\begin{proof}
By  \eqref{eq:Rnspinc} and \eqref{eqSrhoeta1}, we have  an identity in $R(K_{M})$, 
\begin{align}\label{eqetabetaN}
	\bigoplus_{\beta\in 
		\fb^{*}}\eta_{\beta|K_{M}}=
		\Lambda^{\scriptscriptstyle\bullet }\(\fn^{*}_{\bC}\)_{|K_{M}}\otimes \rho_{|K_{M}}.
\end{align}
By \eqref{eq:mpk1}, \eqref{eqetahat}, \eqref{eqetabetaN} and 
Proposition \ref{propnng} \ref{nng6}), we get
\begin{align}\label{eqetabetaN1}
\bigoplus_{\beta\in 
		\fb^{*}}\widehat{\eta}_{\beta|K_{M}}=	\bigoplus_{\beta\in 
	\fb^{*}}\Lambda^{\scriptscriptstyle\bullet 
	}\(\fp_{\fm,\bC}^{*}\)\widehat{\otimes} \eta_{\beta}=\bigoplus_{i=1}^{m}(-1)^{i-1}i 
		\Lambda^{i}\(\fp^{*}_{\bC}\)_{|K_{M}}\otimes \rho_{|K_{M}}.
\end{align} 
Since the restriction $R(K)\to R(K_{M})$ is injective, from 
\eqref{eqetabetaN1},  we get \eqref{eqsumbeta}.
\end{proof}

\subsection{The Selberg zeta function  
$Z_{\eta_{\beta}}$}\label{sZetaetabeta}
By Proposition \ref{propkey1} \ref{keyp2}), $\eta_{\beta}$ 
satisfies Assumption \ref{ass}.  
Let $Z_{\eta_{\beta}}$ be the associated Selberg zeta 
function.  We have an analogue of \cite[Theorem 7.7]{Shfried} and of 
Proposition \ref{propd1RZ}. 

\begin{prop}\label{cor:R}
  The following identity of meromorphic function on $ \bC$ holds,
  \begin{align}\label{eq:RbyZ}
  R_\rho(\sigma)= \(Z_{\eta_{0}}(\sigma)\)^{-1}
 \prod_{\beta\in \fb^{*}_{+}} \(
  Z_{\eta_{\beta}}\(\sigma+|\beta|\)Z_{\eta_{\beta}}\(\sigma-|\beta|\)\)^{-1}.
\end{align}
\end{prop}
\begin{proof}
	This is a consequence of \eqref{defRrho}, \eqref{eqdefsel}, and \eqref{eqAS333}.
\end{proof}

	As in \eqref{eqsumbetamad1}, by \eqref{eqseta}, \eqref{eq:detfor},  and \eqref{eqAS2},  
\begin{align}\label{eqsumbetama}
	Z_{\eta_{\beta}}\(-\sqrt{\sigma^2+|\beta|^2}\)Z_{\eta_{\beta}}\(\sqrt{\sigma^2+|\beta|^2}\)=\[{\rm det}_{\rm gr}\(C^{\fg,Z,\widehat{\eta}_{\beta}}-C^{\fu,\rho}+\sigma^{2}\)\]^{2}.
\end{align}
is a meromorphic function on $\bC$. The following 
 proposition is an analogue of \cite[Theorem 7.8]{Shfried} and of  Theorem \ref{prop:RT11}.

\begin{thm}
  \label{prop:RT}
  The following identity of meromorphic functions on $\bC$ holds, 
  \begin{align}\label{eq:RT}
    R_\rho(\sigma)=T\(\sigma^2\)  \exp\(-\vol(Z)P_{\eta_{0}}(\sigma)\) \prod_{\beta\in \fb^{*}_{+}}  
\frac{Z_{\eta_{\beta}}\(-\sqrt{\sigma^2+|\beta|^2}\)Z_{\eta_{\beta}}\(\sqrt{\sigma^2+|\beta|^2}\)}{Z_{\eta_{\beta}}\big(\sigma+|\beta|\big)Z_{\eta_{\beta}}\big(\sigma-|\beta|\big)}.
  \end{align}
\end{thm}
\begin{proof}
	By \eqref{eqTsigma}, \eqref{eq:D=C-C},  \eqref{eqAS771}, and 
	\eqref{eqsumbeta}, 	the statement of \eqref{eq:detfortor11} still 
	holds in the current 	situation. The rest part of the proof is 
	exactly the same as in the 	proof of Theorem \ref{prop:RT11}. 
\end{proof}

\begin{re}\label{rerCr}
Define $r_{\eta_{\beta}}$, $C_{\rho}$, and $r_{\rho}$ by the same 
formula as in \eqref{eqrrhobeta} and \eqref{eqTTb2}. Proceeding as 
\cite[(7-76)-(7-78)]{Shfried}, using Theorem \ref{prop:RT} instead of 
\cite[Theorem 7.8]{Shfried}, we get \eqref{eq1rd} when $\delta(G)=1$ and $Z_{G}$ is compact.
\end{re}
%
%
%

\section{A cohomological formula for 
$r_{\eta_{\beta}}$}\label{S:rep}

The purpose of this section is to show \eqref{eq2rd} when 
$\delta(G)=1$ and $Z_{G}$ is compact. Its proof 
relies on some deep results from the classification of unitary 
representations of real reductive groups.

This section is organised as follows. In Section 
\ref{sec:rep}, we recall the definition of the infinitesimal 
characters of a $\mathscr U(\fg_{\bC})$-modules, the Harish-Chandra $(\fg_{\bC},K)$-modules, and a relation between the infinitesimal 
character and the  vanishing of $(\fg,K)$-cohomology of a unitary 
Harish-Chandra $(\fg_{\bC},K)$-module, which is due to 
Vogan-Zuckermann \cite{VoganZuckerman}, Vogan \cite{Vogan2}, and 
Salamanca-Riba \cite{Salamanca}. The latter is our 
essential algebraic input in the proof of \eqref{eq2rd}. 

In Section \ref{secbetti}, we obtain a  formula relating  
$H^{\scriptscriptstyle\bullet }(Z,F)$ and the $(\fg,K)$-cohomology of 
certain Harish-Chandra $(\fg_{\bC},K)$-modules. 

Finally, in Section \ref{secnonbetti}, we deduce a similar  formula for $r_{\eta_{\beta}}$ and we prove \eqref{eq2rd}.

We use the notation in Sections \ref{Sreductive} and \ref{Sselberg}. In Sections \ref{sec:rep} and \ref{secbetti},  we  assume neither $\delta(G)=1$ nor $Z_{G}$ is compact. 
\subsection{Some results from representation theory}\label{sec:rep}
We recall some basic facts on the representation theory of real reductive groups.  

\subsubsection{Infinitesimal characters}
 A morphism of algebras 
$\chi:\mathscr{Z}(\fg_\bC)\to \bC$ will be called a character of 
$\mathscr Z(\fg_\bC)$. Clearly, for $a\in \bC$, we have
\begin{align}
	\chi(a)=a. 
\end{align}

By \eqref{eqZcZc}, $\mathscr Z(\fg_{\bC})$ 
is equipped with a complex conjugation. Moreover,  the Cartan involution 
$\theta$ extends to complex automorphism on 
$\mathscr Z(\fg_{\bC})$. Also, the anti-automorphism $z \to  z^{\rm tr}$ of $\mathscr{U}(\fg_{\bC})$ \cite[Proposition 3.7]{KnappLie}, induced by $a\in \fg\to -a\in \fg$, 
descends to a complex automorphisms of $\mathscr Z(\fg_{\bC})$.  For $z\in \mathscr Z(\fg_{\bC})$, set 
\begin{align}
&	\ol{\chi}(z)=\ol{\chi\(\ol{z}\)},&\chi^{\theta}(z)=\chi(\theta 
z),& &\chi^{\rm tr}(z)=\chi\(z^{\rm tr}\). 
\end{align} 
Then, $\ol{\chi}$, $\chi^{\theta}$, and  $\chi^{\rm tr}$ are characters of $\mathscr 
Z(\fg_{\bC})$.


\begin{defin}
A complex representation of $\fg_\bC$ is said to have  infinitesimal character $\chi$,
if $z\in \mathscr Z(\fg_\bC)$ acts as a scalar $\chi(z)\in\bC$.

A complex representation of $\fg_\bC$ is said to have  generalised infinitesimal character $\chi$,
if for there is $i\gg1$ such that for all $z\in \mathscr 
Z(\fg_{\bC})$,  $(z-\chi(z))^i$ acts like $0$. 
\end{defin}

If $W$ is a complex representation of $\fg$ with infinitesimal 
character $\chi_{W}$, then $\ol{W}$,  $W^{\theta}$ (defined in an 
obvious way as \eqref{rt=rt}), and $W^{*}$ have  infinitesimal 
characters  $\chi_{\ol{W}}$, $\chi_{W^{\theta}}$, and $\chi_{W^{*}}$, so that 
\begin{align}\label{eqchiWd}
&\chi_{\ol{W}}=\ol{\chi}_{W}, & \chi_{W^{\theta}}=\chi_{W}^{\theta}, 
&&	\chi_{W^{*}}=\chi_{W}^{\rm tr}. 
\end{align}
Therefore, if $W$ is a unitary representation of $\fg$, we have
\begin{align}\label{eqchiWcd}
\ol{\chi}^{\rm tr}_{W}=\chi_{W}.
\end{align}
If $W$ is a representation of $\fg$ which is equipped with  an admissible Hermitian metric, then $\ol{W}^{*}\simeq 
W^{\theta}$, so that 
\begin{align}\label{eqchiWdd}
	\ol{\chi}_{W}^{\rm tr}=\chi_{W}^{\theta}. 
\end{align}

Let us recall the definition of the Harsh-Chandra parameter for  a 
character of $\mathscr Z(\fg_{\bC})$. Let $\fh_{\bC}\subset \fg_{\bC}$ be a Cartan subalgebra of 
$\fg_{\bC}$. Let $S(\fh_{\bC})$ be the symmetric algebra of 
$\fh_{\bC}$. If  $W(\fh_{\bC}:\fg_{\bC})$ denotes  the  algebraic Weyl group, let  $S(\fh_{\bC})^{W(\fh_{\bC}:\fg_{\bC})}\subset S(\fh_{\bC})$ be the 
$W(\fh_{\bC}:\fg_{\bC})$-invariant subalgebra of $S(\fh_{\bC})$.
 Let
\begin{align}
  \phi_{\rm HC}:\mathscr Z(\fg_\bC)\simeq S(\fh_{\bC})^{W(\fh_{\bC}:\fg_{\bC})}
\end{align}
be the Harish-Chandra isomorphism \cite[Section V.5]{KnappLie}. For 
$\Lambda\in \fh^*_{\bC}$, we can associate to it a character 
$\chi_\Lambda$ of $\mathscr Z(\fg_\bC)$ as follows: for $z\in \mathscr Z(\fg_\bC)$,\index{C@$\chi_{\Lambda}$}
\begin{align}
  \chi_\Lambda(z)=\<\phi_{\rm HC}(z), \Lambda\>.
\end{align}
By \cite[Theorem 5.62]{KnappLie}, every character of $\mathscr Z(\fg_\bC)$ 
is of the form $\chi_\Lambda$, for some $\Lambda\in \fh^*_{\bC}$. 
Also, $\Lambda$ is uniquely determined up to an action of 
$W(\fh_{\bC}: \fg_{\bC})$. Such an element $\Lambda\in \fh^*_{\bC}$ is called  the Harish-Chandra parameter of the character.



\subsubsection{Harish-Chandra $(\fg_\bC,K)$-module and its $(\fg,K)$-cohomology}
\begin{defin}
 A complex $\mathscr U(\fg_\bC)$-module $V$, equipped with an  action of $K$,  is called a Harish-Chandra $(\fg_\bC,K)$-module, if
 \begin{enumerate}
  \item the space $V$ is a finitely generated $\mathscr U(\fg_\bC)$-module;
  \item every $v\in V$ is $K$-finite, i.e., $\{k\cdot v\}_{k\in K}	$ spans a finite dimensional vector space;
  \item the actions of $\fg_\bC$ and $K$ are compatible;
  \item each irreducible $K$-module occurs only for a finite number of times in $V$.
\end{enumerate}
\end{defin}


Let $\widehat{G}_u$ be the unitary dual of $G$, that is  the set of 
equivalence classes  of complex irreducible unitary representations 
$\pi$ of $G$ on Hilbert spaces $V_\pi$. For $(\pi,V_{\pi})\in 
\widehat{G}_u$, let $V_{\pi,K}$ be the space of $K$-finite vectors.  
By \cite[Theorem 8.1, Proposition 
 8.5]{Knappsemi}, $\fg_\bC$ acts on $V_{\pi,K}$ such that 
 $V_{\pi,K}$ is a Harish-Chandra $(\fg_\bC,K)$-module.

%

If $V$ is a Harish-Chandra $(\fg_\bC,K)$-module, let 
$H^{\scriptscriptstyle\bullet }(\fg,K;V)$ \index{H@$H^{\scriptscriptstyle\bullet }(\fg,K;V)$} be the 
$(\fg,K)$-cohomology of $V$ \cite[Section I.1.2]{BW}. 


\begin{thm}\label{thm:vankey}
Let $V$ be a Harish-Chandra $(\fg_\bC,K)$-module with generalised infinitesimal  character $\chi$.
Let $W$ be a finite dimensional $(\fg_\bC,K)$-module with infinitesimal 
character $\chi_{W}$.  If $\chi\neq \chi_{W}^{\rm tr}$, then
\begin{align}\label{eq:key1}
   H^{\scriptscriptstyle\bullet }(\fg,K;V\otimes W)=0.
\end{align}
\end{thm}
\begin{proof}
	If $V$ has an  infinitesimal  character, the theorem is due to 
	\cite[Theorem I.5.3(ii)]{BW}. If $V$ has a generalised 
	infinitesimal  character, a proof can be found in  \cite[Theorem 
	8.8]{Shfried}.
\end{proof}

The converse of \eqref{eq:key1} is not true in general. But it still 
holds if $V$ is unitary. 


\begin{thm}\label{thm:vankey1}Let $W$ be a finite dimensional 
	$(\fg_\bC,K)$-module with infinitesimal character $\chi_{W}$.
If  $(\pi,V_\pi)\in \widehat{G}_u$, then
\begin{align}\label{eq:key2}
 \chi_{\pi}\neq\chi_{W}^{\rm tr} \iff H^{\scriptscriptstyle\bullet }(\fg,K;V_{\pi,K}\otimes W)=0.
\end{align}
\end{thm}
\begin{proof}
	  The  direction $\implies$ of \eqref{eq:key2} is \eqref{eq:key1}.
The  direction  $\impliedby$  of \eqref{eq:key2} is a consequence of 
Vogan-Zuckerman \cite{VoganZuckerman}, Vogan \cite{Vogan2}, and  
Salamanca-Riba \cite{Salamanca}. Indeed, if 
$\chi_{\pi}=\chi_{W}^{\rm tr}$, then the Harish-Chandra parameter of 
$\chi_{\pi}$ is stronger regular in the sense of \cite[p. 
525]{Salamanca}. Such representation $\pi$ is classified in 
\cite{Salamanca}, which has non vanishing $(\fg,K)$-cohomology by \cite{VoganZuckerman, Vogan2}.\end{proof}

\subsubsection{ Root system and $\fn$-homology} \label{ssRS}
We use the notation in Section \ref{sCartanF}.  Let $\fh=\fb\oplus \ft$ be the fundamental Cartan subalgebra 
of $\fg$.  Let $R\subset \fb^{*}\oplus \sqrt{-1}\ft^{*}$ be a root 
system of $(\fh,\fg)$. By \cite[Proposition 11.16]{KnappLie} (see also 
\cite[Proposition 3.7]{BS19}), there are no  real roots in $R$.  Let $R^{\rm im}$ and $R^{\rm c}$ be 
the systems of imaginary roots and complex roots, so that 
\begin{align}
	R=R^{\rm im}\sqcup R^{\rm c}.
\end{align} 
Then, $R^{\rm im}$ is a root system of 
$(\fh,\fz(\fb))$. Also, $R^{\rm im}|_{\ft}$ is a root system of 
$(\ft,\fm)$. 

We fix a positive root system $R_{+}\subset R$. Set 
\begin{align}
&	R^{\rm im}_{+}=R^{\rm im}\cap R_{+},& R^{\rm c}_{+}=R^{\rm 
c}\cap R_{+}. 
\end{align} 
As explained in \cite[Section 3.5]{BS19}, we can choose $R_{+}$ such that $R^{\rm c}_{+}$ is stable 
under complex conjugation. 

Set\index{R@$\rho^{\fk}$}
\begin{align}\label{eq:rhok88}
  \varrho^\fu=\frac{1}{2}\sum_{\alpha\in 
  R^{+}}\alpha\in \fb^{*}\oplus \sqrt{-1}\ft^{*}. 
\end{align}
 By Kostant's strange formula \cite{Kostant76} or \cite[Proposition 7.5.1]{B09}, we have
 \begin{align}\label{eq:kostant}
   \left|\varrho^\fu\right|^2=-\frac{1}{24}\Tr^{\fu}\[C^{\fu,\fu}\].
 \end{align}
Define $\varrho^{\fu(\fb)} \in \fb^{*}\oplus \sqrt{-1}\ft^{*}$ and 
$\varrho^{\fu_{\fm}}\in \sqrt{-1}\ft^{*}$ in the same way, which are 
associated to $R^{\rm 
im}_{+}$ and $R^{\rm im}_{+|\ft}$. Then, 
\begin{align}\label{equbum}
&\varrho^{\fu(\fb)}=\(0,\varrho^{\fu_\fm}\), 
&\varrho^{\fu}_{|\ft}=\varrho^{\fu_\fm}. 
\end{align}

%

%
%


If $V$ is a Harish-Chandra $(\fg_{\bC},K)$-module, denote by $H_{\scriptscriptstyle\bullet }(\fn,V)$ its 
$\fn$-homology. By \cite[Proposition 2.24]{HechtSchmid}, 
$H_{\scriptscriptstyle\bullet }(\fn,V)$ \index{H@$H_\cdot(\fn,V)$} is a Harish-Chandra  
$(\fm_{\bC}\oplus\fb_{\bC}, K_M)$-module. 
 If $V$ possesses  an infinitesimal character with Harish-Chandra 
 parameter $\Lambda\in \fh^*_{\bC}$, by \cite[Corollary 
 3.32]{HechtSchmid},  $H_{\scriptscriptstyle\bullet }(\fn,V)$ can be decomposed into a 
finite direct sum of Harish-Chandra $(\fm_{\bC}\oplus\fb_{\bC}, 
K_M)$-modules  whose   generalised infinitesimal characters are given by
\begin{align}\label{eqinHn}
\chi_{w\Lambda+\varrho^\fu-\varrho^{\fu(\fb)}},
\end{align}
for some $w\in W(\fh_{\bC}:\fg_{\bC})$.

\subsection{The cohomology of $H^{\scriptscriptstyle\bullet }(Z,F)$}\label{secbetti}
We use the notation in Section \ref{sFadmetric}. 
Recall that $\Gamma\subset G$ is a discrete cocompact torsion free 
subgroup of $G$ and that  $\rho: G\to \GL(E)$ is a representation of $G$ 
with an  admissible metric. Let $(F,g^{F})$ be the associated 
Hermitian flat vector bundle.

By \cite[p.~23, Theorem]{GMP}, we can decompose $L^2(\Gamma\backslash G)$ into a
direct Hilbert  sum of countable irreducible unitary representations of $G$,
\begin{align}\label{eq:Geldecomp}
 L^2\(\Gamma\backslash G\)=\bigoplus_{\pi\in 
 \widehat{G}_u}^{\mathrm{Hil}}n(\pi)V_\pi,
\end{align}
with $n(\pi)<\infty$.

For any unitary representation $(\tau,E_{\tau})$ of $K$, since $ 
C^{\fg,Z,\tau}$ is elliptic and self-adjoint and since $Z$ is 
compact,  we have a finite sum
\begin{align}\label{eqKerCpi}
	\ker \(C^{\fg,Z,\tau}-\lambda \)=\bigoplus_{ { \pi\in 
 \widehat{G}_u, \chi_{\pi}(C^{\fg})=\lambda}} 
 n(\pi)\(V_{\pi,K}\otimes E_{\tau}\)^{K}.
\end{align}

Let $\cX(\rho^{*})$ be set of the infinitesimal 
characters of all  irreducible subrepresentations of $\rho^{*}$. Note that by Remark \ref{regGsub}, the sets of all  irreducible 
subrepresentations of $\rho^{*}$ of the group $G$ and of the Lie 
algebra $\fg$ coincide. 


\begin{thm}\label{thmHZF009}
	We have
	\begin{align}\label{eqHZF}
		H^{\scriptscriptstyle\bullet }(Z,F)=\bigoplus_{ \pi\in 
 \widehat{G}_u,\chi_{\pi}\in \cX({\rho^{*}})} 
 n(\pi)H^{\scriptscriptstyle\bullet }\(\fg,K;V_{\pi,K}\otimes E\).
	\end{align}
If $H^{\scriptscriptstyle\bullet }(Z,F)=0$, then for any $\pi\in \widehat{G}_u$ such that 
$\chi_\pi\in \cX({\rho^{*}})$, we have
\begin{align}\label{eq:mushiva}
  n(\pi)=0.
\end{align}	
If $\rho$ is irreducible such that $\rho^{\theta}\neq \rho$, then 
\begin{align}\label{eqHZF=0}
H^{\scriptscriptstyle\bullet }(Z,F)=0.
\end{align}
\end{thm}
\begin{proof}
	Since the $G$-representation with an admissible metric is completely reducible, we 
	can assume that $\rho$ is irreducible with the infinitesimal 	
	character $\chi_{\rho}$. 	
	
	By \eqref{eqchiWd}, we have 
	$\chi_{\rho}^{\rm tr}(C^{\fg})=\chi_{\rho}(C^{\fg})=C^{\fg,\rho}$. 
	By \eqref{eqHo1}, \eqref{eq:D=C-C} and \eqref{eqKerCpi}, we  have
		\begin{align}\label{eqHoLie}
		H^{\scriptscriptstyle\bullet }(Z,F)=\bigoplus_{ { \pi\in 
 \widehat{G}_u, \chi_{\pi}(C^{\fg})=\chi_{\rho}^{\rm tr}(C^{\fg})}} 
 n(\pi)\(V_{\pi,K}\otimes 
 \Lambda^{\scriptscriptstyle\bullet }(\fp^{*}_{\bC})\otimes E\)^{K}.
	\end{align}
When $\chi_{\pi}(C^{\fg})=\chi_{\rho}^{\rm tr}(C^{\fg})$, by  Hodge theory 
for Lie algebras \cite[Proposition II.3.1]{BW}, we have
\begin{align}\label{eqHoLie2}
	\(V_{\pi,K}\otimes 
 \Lambda^{\scriptscriptstyle\bullet }(\fp^{*}_{\bC})\otimes 
 E\)^{K}=H^{\scriptscriptstyle\bullet }(\fg,K;V_{\pi,K}\otimes E). 
\end{align}
By \eqref{eq:key1}, \eqref{eqHoLie}, and \eqref{eqHoLie2}, we get \eqref{eqHZF}. 
By \eqref{eq:key2} and \eqref{eqHZF}, we get 
\eqref{eq:mushiva}. 

To show \eqref{eqHZF=0}, it is enough to show that if 
$\rho^{\theta}\neq \rho$, then for all $\pi\in 
\widehat{G}_{u}$, we have  
\begin{align}
	\chi_{\pi}\neq \chi_{\rho}^{\rm tr}. 
\end{align}
Otherwise there is $\pi\in \widehat{G}_{\pi}$ such that 
$\chi_{\pi}=\chi_{\rho}^{\rm tr}$. Using $ \ol{\pi}^{*}\simeq \pi$ and $\ol{\rho}^{*}\simeq\rho^{\theta}$, by \eqref{eqchiWcd} and 
\eqref{eqchiWdd}, we have 
\begin{align}\label{eqchi22}
\chi_{\rho}=\chi_{\pi}^{\rm tr}=\ol{\chi}_{\pi}=\ol{\chi}_{\rho}^{\rm tr}=\chi_{\rho}^{\theta}.
\end{align}
Since  $\rho$ and  $\rho^{\theta}$ are irreducible and  have finite 
dimensions, Equation \eqref{eqchi22} implies $\rho\simeq \rho^{\theta}$,  which is  a contradiction with our assumption. 
\end{proof}

\begin{re}	
	Equations \eqref{eqHZF} and 
	\eqref{eqHZF=0} are   \cite[Theorems VII.6.1 and VII.6.7]{BW}. Equation \eqref{eqHZF} is originally due to Matsushima 
	\cite{MatsushimaBetti} where  $\rho$ is supposed to be
	 trivial.  
	\end{re}

\subsection{A formula for $r_{\eta_{\beta}}$}\label{secnonbetti}
Assume now that $\delta(G)=1$ and $Z_{G}$ is compact, and that 
$\rho:G\to \GL(E)$ is a 
$G$-representation with an admissible metric  such that $\rho\simeq 
\rho^{\theta}$ and that $C^{\fu,\rho}\in \bR$ is a scalar. 

By \cite[Corollary 8.15]{Shfried}, we have
\begin{multline}\label{eq814}
	r_{\eta_{\beta}}=\frac{1}{\chi(K/K_M)}\sum_{\tiny \substack{\pi\in 
	\widehat{G}_u,\chi_{\pi}(C^{\fg})=C^{\fu,\rho}\\ 0\l i\l \dim 
	\fp_{\fm}\\ 0\l j\l 2\ell }}
(-1)^{i+j} n(\pi) \Big(\dim  H^i\big(\fm, K_M; H_j(\fn,V_{\pi,K})\otimes 
E_{\eta_{\beta}}^{+}\big)\\-\dim H^i\big(\fm, K_M; H_j(\fn,V_{\pi,K})\otimes 
E_{\eta_{\beta}}^{-}\big)\Big) .
\end{multline}


\begin{prop}\label{prop:vani3}
Let $(\pi,V_\pi)\in \widehat{G}_u$. Assume that 
$\chi_\pi(C^{\fg})=C^{\fu,\rho}$ and
\begin{align}\label{eq:Hn0}
\bigoplus_{\tiny \substack{0\l i\l \dim \fp_{\fm}\\ 0\l j\l 2\ell \\ s\in 
\{\pm\} }}H^i\big(\fm, K_M; H_j(\fn,V_{\pi,K})\otimes 
E_{\eta_{\beta}}^{s}\big)\neq0.
\end{align}
Then,
\begin{align}
\chi_\pi\in \cX(\rho^{*}).
\end{align}
\end{prop}
\begin{proof}
We use the notation in Section \ref{ssRS}.  Let $\Lambda(\pi^{*})\in \fh^*_\bC$ be a Harish-Chandra parameter of 
$V_{\pi^{*}\!\!,\,K}$. We need to show that there is $w\in 
W(\fh_\bC:\fg_{\bC})$ and a Harish-Chandra parameter  
$\Lambda(\rho)\in \fb^{*}\oplus \sqrt{-1}{\ft}^{*}$ of an irreducible 
$\fg$-submodule   of $\rho$, such that 
\begin{align}\label{eq:hou}
  w\Lambda(\pi^{*})=\Lambda(\rho).
\end{align}

Recall that $B^*$ is the bilinear form on $\fg^*$ induced by $B$. It extends to $\fg^*_\bC$ in an obvious way. 
Since $\chi_\pi(C^{\fg})=C^{\fu,\rho}$, using Harish-Chandra 
isomorphism (see \cite[Example 5.64]{KnappLie}), we have
\begin{align}\label{eq:C=0}
  B^*\(\Lambda(\pi^{*}),\Lambda(\pi^{*})\)-B^{*}\(\varrho^{\fu},\varrho^{\fu}\)=C^{\fu,\rho}.
\end{align}

%

%
By \eqref{eq:key1},  \eqref{equbum}, \eqref{eqinHn}, and \eqref{eq:Hn0}, there exist 
$w\in W(\fh_\bC:\fg_{\bC})$, $w'\in W(\ft_{\bC}:\fm_{\bC})\subset 
W(\fh_\bC:\fg_{\bC})$ 
and the highest weight $\mu_{\beta}\in \sqrt{-1}\ft^*$ of an 
irreducible $(\fm_{\bC},K_{M})$-submodule of $ 
\eta_{\beta}^{+}\oplus \eta_{\beta}^{-}$ such that
\begin{align}\label{eq:sf3}
	w\Lambda(\pi^{*})_{|\ft_\bC}=w'\(\mu_{\beta}+\rho^{\fu_\fm}\).
\end{align}

By Proposition \ref{prop:nnam1} and \eqref{eq:rhok88}, we have 
 \begin{align}\label{eqrhobb}
\varrho^{\fu}=\varrho^{\fu(\fb)}+\(\ell \alpha_{0},0\).	
\end{align} 
By \eqref{eqlla0}, \eqref{eqAS2}, \eqref{equbum}, \eqref{eq:C=0}-\eqref{eqrhobb}, there exists 
$w''\in  W(\fh_\bC:\fg_{\bC})$ such that
\begin{align}\label{eq:Cpi=0}
	w''\Lambda(\pi^{*})=\(\pm\beta, 
  \mu_{\beta}+\rho^{\fu_\fm}\)=\(\pm\beta, 
  \mu_{\beta}\)+\rho^{\fu(\fb)}. 
\end{align}
In particular, $\Lambda(\pi^{*})\in \fb^{*}\oplus \sqrt{-1}\ft^{*}$.


By \eqref{eqdetapm}, $\big(\beta, \mu_\beta\big)\in 
\fb^*\oplus \sqrt{-1}\ft^{*}$ is a highest
 weight of an irreducible $(\fm_\bC\oplus 
\fb_\bC,K_{M})$-submodule  of $\ker D^{S^{\fu^\bot(\fb)}\otimes E}$.
By \cite[Theorem 4.2.2]{HuangDiracCoh}, there exists $w_1\in 
W(\fh_\bC:\fg_{\bC})$ and  a Harish-Chandra parameter 
$\Lambda(\rho)\in \fh_{\bC}^{*}$ of an 
irreducible $\fg$-submodule  of $\rho$,  such that
\begin{align}\label{eq:la2}
  \(\beta, \mu_{\beta}\)=w_{1}\Lambda(\rho)-\rho^{\fu(\fb)}. 
\end{align}

By  \eqref{eqdetapm} and \eqref{eqAS771}, $(-\beta, \mu_\beta)\in \fb^*\oplus \sqrt{-1}\ft^{*}$ is 
also the highest weight of an irreducible   $(\fm_\bC\oplus 
\fb_\bC,K_{M})$-submodule of $\ker D^{S^{\fu^\bot(\fb)}\otimes E}$. 
As before,  there exists $w_2\in W(\fh_\bC:\fg_{\bC})$ and a Harish-Chandra parameter $\Lambda(\rho)\in \fh_{\bC}^{*}$ of an 
irreducible $\fg$-submodule of $\rho$, such that
\begin{align}\label{eq:la3}
  \(-\beta, \mu_{\beta}\)=w_2\Lambda(\rho)-\rho^{\fu(\fb)}.
\end{align}
By \eqref{eq:Cpi=0}-\eqref{eq:la3}, we get \eqref{eq:hou}.
\end{proof}

\begin{cor}\label{cor:forr1}
 For $\beta\in \fb^{*}$, we have
\begin{multline}\label{eq734}
  r_{\eta_{\beta}}=\frac{1}{\chi(K/K_M)}\sum_{\tiny \substack{\pi\in 
  \widehat{G}_{u},
  \chi_\pi\in \cX({\rho^{*}})\\ 0\l i\l \dim \fp_{\fm}\\ 0\l j\l 
  2\ell }  } (-1)^{i+j}n(\pi)
  \Big(\dim
  H^i\(\fm,K_M;H_j(\fn,V_{\pi,K})\otimes  E_{\eta_\beta}^{+}\)\\-\dim 
  H^i\(\fm,K_M;H_j(\fn,V_{\pi,K})\otimes  E_{\eta_\beta}^{-}\)\Big).
\end{multline}
In particular, if $H^{\scriptscriptstyle\bullet }(Z,F)=0$, then for all $\beta\in\fb^{*}$, we have
\begin{align}\label{eq735}
	r_{\eta_{\beta}}=0.
\end{align}
\end{cor}
\begin{proof}
By \eqref{eq814} and Proposition 
	\ref{prop:vani3}, we get \eqref{eq734}. From \eqref{eq:mushiva} and 
	\eqref{eq734}, we get  \eqref{eq735}. 
\end{proof}

\begin{re}
		By \eqref{eqTTb2}, \eqref{eq735}, and by Remark \ref{rerCr}, 
		we get \eqref{eq2rd} in the case where $\delta(G)=1$ and $Z_{G}$ is 
		compact. We finish the proof of Theorem \ref{thm1} in full 
		generality. 
\end{re}

\section{An extension to orbifolds}\label{Snontorsionfree}
In this section, we  no longer assume $\Gamma\subset G$ is 
torsion free. Then $Z=\Gamma\backslash G/K$ is a closed Riemannian 
orbifold with Riemannian metric $g^{TZ}$. 
Let us indicate the essential steps in generalising the previous results 
to orbifolds.



If $\gamma\in \Gamma$,  $\Gamma(\gamma)$ is not always torsion free. 
The cardinality 
\begin{align}
	\left|\ker\(\Gamma(\gamma)\to {\rm Diffeo} 
	(Z(\gamma)/K(\gamma))\)\right|
\end{align} 
depends only on the conjugacy class $[\gamma]$ and will be denoted by 
$n_{[\gamma]}$. We define $B_{[\gamma]}$  by the same 
formula as in \eqref{eqBgamma}. 
 By \cite[Proposition 5.3]{Shen_Yu}, we have
\begin{align}
\frac{	\vol(\Gamma(\gamma)\backslash 
Z(\gamma))}{\vol(K(\gamma))}=\frac{\vol\(B_{[\gamma]}\)}{n_{[\gamma]}}. 
\end{align}

By \cite[Remark 
5.6, (5.59)]{Shen_Yu}, as in \eqref{DKVB}, the  closed geodesics (see \cite{GuruprasadHaefiger06} or \cite[Remark 2.26]{Shen_Yu}) on the 
orbifold $Z$ with positive length are given by 
\begin{align}
	\coprod_{[\gamma]\in [\Gamma_{+}]} B_{[\gamma]}.
\end{align} 
For $[\gamma]\in [\Gamma_{+}]$, all the elements of $B_{[\gamma]}$ have the same length $\ell_{[\gamma]} > 0$.

For $[\gamma]\in [\Gamma_{+}]$, the group $\mathbb{S}^{1}$ acts locally freely on the orbifold 
$B_{[\gamma]}$ by rotation,  so that  $B_{[\gamma]}/\mathbb{S}^{1}$ is still a closed  orbifold.  Set
\begin{align}\label{eqdefm}
	m_{[\gamma]}=n_{[\gamma]} \left|\ker\(\bbS^{1}\to 
	\rm{Diffeo}(B_{[\gamma]}) \)\right|\in \mathbf{N}^{*}. 
\end{align}

If $\rho:\Gamma\to \GL_{r}(\bC)$ is a representation of $\Gamma$, for $\Re(\sigma)\gg1$ large enough, we define Ruelle's dynamical zeta 
function $R_{\rho}(\sigma)$ by the same formula  \eqref{defRrho} with $m_{[\gamma]}$ 
defined by \eqref{eqdefm}.  As before, when $\delta(G)\g 2$,  
\begin{align}
	R_{\rho}(\sigma)\equiv1. 
\end{align}
By \cite[Theorem 7.3]{Shen20}, if $\dim Z$ is odd, $R_{\rho}(\sigma)$ has a meromorphic extension to $\sigma\in \bC$.  

Let $\rho:G\to \GL(E)$ be a finite  dimensional complex representation of $G$ with an admissible metric. Let $F$ be the   orbifold flat vector 
bundle on $Z$ associated to $\rho_{|\Gamma}$.  As in Section \ref{sFadmetric}, $F$ is 
equipped canonically with a Hermitian metric $g^{F}$. The 
analytic torsion of $F$ associated to $(g^{TZ},g^{F})$ is defined in \cite{Daiyu}, \cite[Section 4.2]{Shen_Yu} (see also \cite{Ma_Orbifold_immersion}). 

\begin{thm}
The statements of Theorems \ref{thmMS0}, \ref{thm1}, and \ref{thm2} 
still hold for orbifolds. In particular, we get Theorem \ref{Thm2}. 
\end{thm}
\begin{proof}
	Using the orbifold trace formula \cite[Theorem 5.4]{Shen_Yu} and 
	\cite[Theorem 5.4]{Ma_bourbaki}, we get the orbifold version of 
	Theorem \ref{thmMS0}.

	The proof of Theorem \ref{thm1} in the case of orbifold is similar as before and we need only consider the case 
	$\delta(G)=1$. We can define the Selberg zeta function  
	by	 the same formula \eqref{eqdefsel} with $m_{[\gamma]}$ 
	defined by \eqref{eqdefm}. By \cite[Section 7.2]{Shen20}, the statement 
	of 			Theorem \ref{thm:detfor} still holds for orbifolds. 
	By exactly the same method, the statements of Proposition 
	\ref{propd1RZ}, Theorem \ref{prop:RT11}, Proposition \ref{cor:R}, and Theorem \ref{prop:RT} hold for orbifold. Using the orbifold Hodge theory (c.f. \cite[Theorem 
	4.1]{Shen_Yu}), we can deduce that the statements of Theorem 
	\ref{thmHZF009} and Corollary \ref{cor:forr1} hold as well. In 
	this way, we get  Theorem \ref{thm1} for orbifolds. 
	
	As in the proof of Theorem \ref{thm2} 
		given in Section \ref{sC2}, 	the orbifold version of 
		Theorem \ref{thm2} is a consequence of 	the 	orbifold 
		version of Theorem \ref{thm1}. The proof of our Theorem is 
		completed. 
	\end{proof}

\def\cprime{$'$}
\providecommand{\bysame}{\leavevmode\hbox to3em{\hrulefill}\thinspace}
\providecommand{\MR}{\relax\ifhmode\unskip\space\fi MR }
\providecommand{\MRhref}[2]{%
  \href{http://www.ams.org/mathscinet-getitem?mr=#1}{#2}
}
\providecommand{\href}[2]{#2}

\end{document}